\documentclass{amsart}
\usepackage{amsmath}
\usepackage{amssymb}
\usepackage{graphicx}

\newcounter{index}

\newcommand{\nat}{\ensuremath{\omega}}
\newcommand{\Tr}{{\rm Tr}}

\newcommand{\ie}{{\rm i.e.,} }
\newcommand{\cf}{{\rm cf.~}}
\newcommand{\n}{\ensuremath{n \in \nat}}
\newcommand{\ep}{\ensuremath{\varepsilon}}
\newcommand{\del}{\ensuremath{\Delta^1_1}}
\newcommand{\sig}{\ensuremath{\Sigma^1_1}}
\newcommand{\pii}{\ensuremath{\Pi^1_1}}
\newcommand\tboldsymbol[1]{%
\protect\raisebox{0pt}[0pt][0pt]{%
$\underset{\widetilde{}}{\boldsymbol{#1}}$}\mbox{\hskip 1pt}}

\newcommand{\om}{\ensuremath{\omega}}
\newcommand{\Q}{\ensuremath{\mathbb Q}}
\newcommand{\R}{\ensuremath{\mathbb R}}
\newcommand{\ds}{\ensuremath{\displaystyle}}

\newcommand{\ca}[1]{\ensuremath{\mathcal{#1}}}
\newcommand{\set}[2]{\ensuremath{\{#1 \hspace{0.3mm} \mid \hspace{0.3mm} #2\}}}
\newcommand\surj{\twoheadrightarrow}
\newcommand{\lh}{{\rm lh}}
\renewcommand{\ie}{\text{i.e.,}~}
\newcommand{\ck}{\ensuremath{\om_1^{CK}}}
\newcommand{\ckr}[1]{\ensuremath{\om_1^{#1}}}

\newcommand{\dec}[1]{\ensuremath{u[#1]}}
\newcommand{\Seq}{\texttt{Seq}}
\newcommand{\kbleq}{\ensuremath{\leq_{KB}}}

\newcommand{\rfn}[1]{\ensuremath{\{#1\}}}
\newcommand{\WO}{{\rm WO}}
\newcommand{\LO}{{\rm LO}}
\newcommand{\IF}{{\rm IF}}
\newcommand{\HYP}{{\sf HYP}}
\newcommand{\tu}[1]{\textup{#1}}
\newcommand{\bolds}{\ensuremath{\tboldsymbol{\Sigma}}}
\newcommand{\boldp}{\ensuremath{\tboldsymbol{\Pi}}}

\newcommand{\cn}[2]{\ensuremath{#1 \ast #2}}
\newcommand{\scode}{\mathrm{G}}
\newcommand{\fcode}{\mathrm{F}}
\newcommand{\ocode}{\mathrm{U}}
\newcommand{\bcodefam}{\mathrm{BC}}
\newcommand{\bcf}{\pi c}
\newcommand{\pointcl}{\Gamma}
\newcommand{\gcode}{\ensuremath{{\rm C}_{\pointcl}}}
\newcommand{\baire}{\ensuremath{{\mathcal{N}}}}
\newcommand{\cantor}{\ensuremath{2^\om}}

\newcommand{\posb}{U}

\newcommand{\pr}{{\rm pr}}
\newcommand{\speffclass}{SP}
\newcommand{\spclass}{\tboldsymbol{\speffclass}}
\newcommand{\spcode}{{\rm SPC}}
\newcommand{\spcfun}[1]{\tau_{#1}}
\newcommand{\spcf}{\tau}
\newcommand{\normsp}[1]{|#1|_{{\rm SP}}}
\newcommand{\normb}[1]{|#1|_{\rm BC}}
\newcommand{\restrict}[2]{#1|#2}
\newcommand{\normR}[1]{\|#1\|}
\newcommand{\strans}{{\rm str}}
\newcommand{\pq}[1]{{\rm q}(#1)}
\newcommand{\dense}[2]{{\rm r}^{#1}(#2)}
\newcommand{\bijN}[1]{[ #1 ]^N}
\newcommand{\pair}[1]{[#1]}
\newcommand{\cgeffclass}{CG}
\newcommand{\cgclass}{\tboldsymbol{\cgeffclass}}
\newcommand{\cgcode}{{\rm CGC}}
\newcommand{\cgcfun}[1]{\pi_{#1}}
\newcommand{\cgcf}{\pi}
\newcommand{\normcg}[1]{|#1|_{{\rm CG}}}
\newcommand{\fprod}{{\rm FinProd}}

\newtheorem{theorem}{Theorem}[section]
\newtheorem{lemma}[theorem]{Lemma}
\newtheorem{definition}[theorem]{Definition}
\newtheorem{proposition}[theorem]{Proposition}
\newtheorem{corollary}[theorem]{Corollary}

\newtheorem{conjecture}[theorem]{Conjecture}

\begin{document}

\title{The Dyck and the Preiss separation uniformly}

\author{Vassilios Gregoriades}

\address{Vassilios Gregoriades\\
Via Carlo Alberto, 10\\ 
10123 Turin, Italy}

\email{vassilios.gregoriades@unito.it}

\date{\today}

\begin{abstract}
We are concerned with two separation theorems about analytic sets by Dyck and Preiss, the former involves the positively-defined subsets of the Cantor space and the latter the Borel-convex subsets of finite dimensional Banach spaces. We show by introducing the corresponding separation trees that both of these results admit a constructive proof. This enables us to give the uniform version of the separation theorems, and derive as corollaries the results, which are analogous to the fundamental fact ``\HYP \ is effectively bi-analytic" provided by the Souslin-Kleene Theorem.
\end{abstract}


\maketitle

\section{Introduction}

The question of realizing a property in a definable uniform way is prominent in effective descriptive set theory. One of the most important examples is the \emph{Souslin-Kleene Theorem}, which says that the separation property of the analytic sets can be realized by a recursive function in the codes. In this article we provide the uniform version of two more separation results about analytic sets, the Dyck and the Preiss separation.

We point out that the proofs of the latter two results that are available in the bibliography (see \cite[28.12;28.15]{kechris_classical_dst} and \cite{preiss_the_convex_generation_of_Borel_sets}) use reduction to contradiction, so we need first to give direct proofs to the Dyck and the Preiss separation. More specifically we will define in each case a well-founded tree, which realizes the separation between the given analytic sets. We point out that the separation in each of these theorems is fundamentally different from the usual separation of analytic sets. More specifically in the Dyck separation, while the inductive step remains the same, we have a smaller choice of sets at the bottom level. On the other hand, in the Preiss separation the bottom level is not very far from the usual one, however the inductive step is more restrictive. These facts are reflected in our separation trees.

We proceed to the presentation of our main results. Concerning the \textbf{Dyck Separation Theorem} we consider for all \n \ the subsets of $\cantor$,
\[
\posb_n = \set{x \in \cantor}{x(n) = 1}.
\]
A set $A \subseteq \cantor$ is \emph{positive} if it belongs to the least family of subsets of $\cantor$, which contains $\set{\posb_n}{\n}$, and is closed under countable unions as well as countable intersections. Similarly a set $A \subseteq \cantor$ is \emph{semi-positive} if it belongs to the least family of sets, which contains $\set{\posb_n}{\n}\cup\{\emptyset,\cantor\}$, and is closed under countable unions/intersections. 

Clearly every positive set is semi-positive, and it is also not hard to verify that every semi-positive set $A \subseteq \cantor$, which is not one of $\emptyset$, $\cantor$ is in fact positive.  Moreover positive sets are different from $\emptyset$, $\cantor$. To see this we remark that every positive set $A$ satisfies $\{{\bf 1}\} \subseteq A \subseteq \cantor \setminus \{{\bf 0}\}$, where ${\bf m}$ denotes the constant sequence $(m,m,m,\dots)$ for $m \in \om$. Hence the positive sets are exactly the semi-positive ones, which are non-empty and have non-empty complement.
 
For $x, y \in \cantor$ we write $x \subseteq y$ if  $\set{\n}{x(n)=1} \subseteq \set{\n}{y(n)=1}$. A set $A \subseteq \cantor$ is \emph{monotone} if for all $x,y \in \cantor$ with $x \in A$ and $x \subseteq y$ we have $y \in A$.

It follows easily that every positive set is monotone. The converse is also true for Borel sets which differ from $\emptyset$ and $\cantor$:

\begin{theorem}[Dyck, \cf \cite{kechris_classical_dst}, Th. 28.11 and 28.12]
\label{theorem Dyck}
Suppose that $A$ and $B$ are disjoint non-empty $\bolds^1_1$ subsets of $\cantor$. If $A$ is monotone then there is a positive Borel set $C \subseteq \cantor$ such that $A \subseteq C$ and $C \cap B = \emptyset$.\smallskip

In particular a Borel subset $A$ of $\cantor$ with $A \not \in \{\emptyset, \cantor\}$ is monotone exactly when it is positive.
\end{theorem}

In the last assertion of the preceding result it is necessary to have that the given set $A$ is Borel, since all positive sets are evidently Borel. It is natural to ask if there are monotone sets, which are not Borel and therefore not positive as well. As it is probably expected the answer is affirmative. We will in fact show a slightly stronger assertion in Proposition \ref{proposition a sig monotone set which is not Borel} below.

In our uniform version of the Dyck separation, we ask for a function which carries codes for the analytic sets $A$ and $B$ to a code for the semi-positive separating set $C$. Here the term ``code" has various interpretations. In the case of the analytic sets, we mean a parameter $\alpha \in \om^\om$ from a fixed universal set. A \emph{Borel code} is a parameter $\alpha \in \om^{< \om}$, which encodes the countable unions and complements, which are necessary to build the corresponding Borel set. The analogous notion of a \emph{semi-positive code} can be given, by using countable unions and intersections. These notions are given explicitly in the sequel (see \ref{subsection Borel codes} and \ref{subsection semi-positive codes} for the latter two).

Our uniform version of the Dyck Separation is

\begin{theorem}
\label{theorem uniformity version Dyck}
There are recursive functions $u,v : \baire \times \baire \to \baire$ such that whenever $\alpha, \beta$ are analytic codes for disjoint sets $A$, $B$ respectively with $A$ being monotone, then there is a semi-positive set $C$ with $A \subseteq C$, $C \cap B = \emptyset$ and such that
\begin{list}{\tu{(}\roman{index}\tu{)}}{\usecounter{index}}
\item $u(\alpha,\beta)$ is a Borel code for $C$; 
\item $v(\alpha,\beta)$ is a semi-positive code for $C$.
\end{list}
\end{theorem}

The proof will be given in Section \ref{section the uniform Dyck separation} of this article.\smallskip

One notable consequence of the Suslin-Kleene Theorem is that $\HYP = \del$, \ie the sets which are obtained by starting with the semi-recursive sets and applying transfinitely the operations of recursive union and complement are the same as the effective bi-analytic (in other words $\del$) sets. Theorem \ref{theorem uniformity version Dyck} has a similar consequence. The effective version of a semi-positive set can be understood in the following two ways: (a) one considers the semi-positive sets, which are \del\;;\; (b) one considers the semi-positive sets, which are constructed by recursive countable unions and intersections using codes just as in the definition of $\HYP$.

We show that the preceding two ways deliver the same class of sets.

\begin{corollary}
\label{corollary effective Dyck}
For all $A \subseteq \cantor$ the following are equivalent:
\begin{list}{\tu{(}\roman{index}\tu{)}}{\usecounter{index}}
\item $A$ is $\del$ and semi-positive;
\item $A$ admits a recursive semi-positive code;
\item $A$ is $\del$ and monotone.
\end{list}
\end{corollary}

Next we move to the \textbf{Preiss Separation Theorem}. A subset $A$ of $\R^N$, where $N \geq 1$, is \emph{convexly generated} if it belongs to the least family, which contains all compact convex sets and is closed under countable increasing unions  as well as countable intersections. As in the case of Borel and semi-positive sets, one can encode the family of convexly generated sets using the \emph{convexly generated codes} - the precise definition is given in \ref{subsection convexly generated codes and the proof to the uniform Preiss Separation Theorem}.

It is evident that all convexly generated subsets of $\R^N$ are Borel. Klee \cite{klee_convex_sets_in_linear_spaces} asked whether the converse is correct, and has answered this affirmatively in the case $N=2$, \cf \cite{klee_convex_sets_in_linear_spacesIII}. Larman \cite{larman_the_convex_sets_in_R3_are_convexly_generated} proved the analogous result in the case $N=3$, and finally Preiss \cite{preiss_the_convex_generation_of_Borel_sets} gave a proof  for arbitrary $N \geq 1$:

\begin{theorem}[Preiss \cite{preiss_the_convex_generation_of_Borel_sets}, see also 28.15 \cite{kechris_classical_dst}]
\label{theorem preiss separation}
Suppose that $A, B$ are disjoint $\bolds^1_1$ subsets of $\R^N$. If $A$ is convex, then there is a convexly generated set $C$ with $A \subseteq C$ and $C \cap B = \emptyset$.\smallskip

In particular a Borel set $A \subseteq \R^N$ is convex if and only if $A$ is convexly generated.
\end{theorem}

It is worth noting the related result of Holick{\'y} \cite{holicky_the_convex_generation_of_convex_Borel_sets_in_locally_convex_spaces} that Theorem \ref{theorem preiss separation} does not extend to the infinite dimensional separable Banach spaces,  where in the definition of convexly generated sets we replace the term ``compact convex" with ``closed convex". 

Unlike the case of Dyck separation, we are not able to obtain a recursive uniformity function, but rather a $\del$-recursive. The reason for this boils down to the fact that we are dealing with subsets of the Euclidean space. We will need for example to consider functions of the form $\left(\alpha \mapsto \text{code for~}H(\fcode^{\R^N}(\alpha) \cap [-m,m]^N)\right)$, and $\left((x,\alpha) \mapsto \ \text{distance between $x$ and~}H(\fcode^{\R^N}(\alpha) \cap [-m,m]^N) \right)$, where $\fcode^{\R^N}(\alpha)$ is the closed set encoded by $\alpha \in \om^{< \om}$ and $H(\cdot)$ is the convex hull, $m \in \om$. These functions are far from being recursive. On the other hand the $\sigma$-compactness of $\R^N$ allows us to keep the effective complexity to the $\del$-level.

\begin{theorem}
\label{theorem uniform preiss separation}
For every natural number $N \geq 1$ there are $\del$-recursive functions $u,v: \baire \times \baire \to \baire$ such that for all $\alpha, \beta \in \baire$, if the analytic sets $\scode_1^{\R^N}(\alpha)$, $\scode_1^{\R^N}(\beta)$ are disjoint and $\scode_1^{\R^N}(\alpha)$ is convex then there is a convexly generated set $C$, which separates $\scode_1^{\R^N}(\alpha)$ from $\scode_1^{\R^N}(\beta)$, and
\begin{list}{\tu{(}\roman{index}\tu{)}}{\usecounter{index}}
\item $u(\alpha,\beta)$ is a Borel code for $C$;
\item $v(\alpha,\beta)$ is a convexly generated code for $C$.
\end{list}
\end{theorem}

The analogous-to-Corollary \ref{corollary effective Dyck} result is

\begin{corollary}
\label{corollary effective Preiss}
For all $A \subseteq \R^N$ the following are equivalent:
\begin{list}{\tu{(}\roman{index}\tu{)}}{\usecounter{index}}
\item $A$ is $\del$ and convexly generated;
\item $A$ admits a $\del$ convexly generated code;
\item $A$ is $\del$ and convex.
\end{list}
\end{corollary}

Notice that we do not claim a \emph{recursive} convexly generated code in the preceding result; this is due to the fact that our uniformity functions are $\del$-recursive. Nevertheless we think that the former is indeed true.

\begin{conjecture}
\label{conjecture recursive convexly generated code}
Every subset of $\R^N$ is $\del$ and convexly generated exactly when it admits a recursive convexly generated code.
\end{conjecture}

\subsection{Notation.}

By $\ca{X}$, $\ca{Y}$ we always mean Polish spaces. As usual in effective descriptive set theory we will write $P(x)$ instead of $x \in P$. The natural numbers are identified with the first infinite ordinal $\om$. By $\om^{< \om}$ we mean the set of all finite sequences of naturals, including the empty one, which we denote by $\emptyset$. Given $u = (u_0,\dots,u_{n-1}) \in \om^{< \om}$ the \emph{length $\lh(u)$ of $u$} is the preceding $n$ - if $n= 0$ then $u = \emptyset$. The \emph{concatenation} $\cn{u}{v}$ of two finite sequences $u,v$ is $(u_0,\dots,u_{\lh(u)-1},v_0,\dots,v_{\lh(v)}-s)$. By $u \sqsubseteq v$ we mean that \emph{$u$ is an initial segment of $v$} or that \emph{$v$ extends $u$}.

We fix the following injective enumeration of $\om^{< \om}$:
\[
\langle \cdot \rangle : \om^{< \om} \to \om: (u_0,\dots,u_{n-1}) \mapsto \langle u_0,\dots, u_{n-1} \rangle = p_0^{u_0+1} \cdot p_1^{u_1 + 1} \cdot p_{n-1}^{u_{n-1}+1},
\]
where $(p_n)_{\n}$ is the increasing enumeration of the natural numbers. By convention $\langle \emptyset \rangle = 1$. We denote by $\Seq$ the image of $\om^{< \om}$ under $\langle \cdot \rangle$, and by $\dec{\cdot}: \Seq \to \om^{<\om}$ the inverse of $\langle \cdot \rangle$. Given $s = \langle u_0,\dots, u_{n-1}\rangle \in \Seq$ and $ i < n$ we let $(s)_i$ be the unique natural $u_i$ as before. If $i \geq n$ or if $s \not \in \Seq$ by $(s)_i$ we mean the number $0$.

Throughout this article we fix the following enumeration $(\pq{s})_{s \in \om}$ of non-negative rational numbers
\[
\pq{s} = \frac{(s)_0}{(s)_1+1}\;.
\]
The \emph{Baire space} is $\baire : = \omega^\omega$ and the \emph{Cantor space} is $\cantor$ both spaces considered with the product topology. The members of $\baire$ and $\cantor$ are usually denoted by lowercase Greek letters $\alpha$, $\beta$, $\gamma$, \dots.

Given $\alpha \in \baire$ and a natural $n$ we put
\begin{align*}
\alpha^\ast =& \ (t \mapsto \alpha(t+1))\\
(\alpha)_n =& \ (i \mapsto \alpha(\langle n,i\rangle))\\
\alpha \upharpoonright n =& \ (\alpha(0),\dots,\alpha(n-1))\\
\overline{\alpha}(n) =& \ \langle \alpha(0),\dots, \alpha(n-1)\rangle.
\end{align*}

The first uncountable ordinal is denoted by $\om_1$. The \emph{Church-Kleene} ordinal (\ie the first non-recursive ordinal) is denoted by $\ck$, and its $\alpha$-relativized version by $\ckr{\alpha}$, where $\alpha \in \baire$.

The members of the Baire space encode countable linear orderings in a natural way. Given $\alpha \in \baire$ we consider the set Field$(\alpha) = \set{n \in \om}{\alpha(\langle n,n \rangle) = 1}$ and the binary relation $\leq_{\alpha} = \set{(n,m)}{\alpha(\langle n,m \rangle) = 1}$. We then define the set $\LO$ \emph{codes of countable linear orderings}
\[
\LO = \set{\alpha \in \LO}{\leq_\alpha \ \text{is a linear ordering on Field}(\alpha)}.
\]
The set of \emph{codes of countable well-orderings} is 
\[
\WO = \set{\alpha \in \LO}{\leq_\alpha \ \text{is a well-ordering on Field}(\alpha)}.
\]
Given $\alpha \in \WO$ we denote by $|\alpha|$ the unique ordinal with order type the one of $\leq_\alpha$.

A \emph{tree} on the naturals is a non-empty set $T \subseteq \om^{\om}$, which is closed downwards under $\sqsubseteq$, \ie \ $(u \in T \ \& \ v \sqsubseteq u) \ \Longrightarrow \ v \in T$. The body $[T]$ of a tree $T$ is the set of its infinite branches \set{\alpha \in \baire}{(\forall n)[\alpha \upharpoonright n \in T]}. We also consider \emph{trees of pairs}, \ie \ non-empty sets of tuples $(u,v) \in \om^{< \om} \times \om^{< \om}$ with $\lh(u) = \lh(v)$, which are closed downwards under $\sqsubseteq$, \ie \ if $(u,v) \in T$ and $(u',v')$ is such that $\lh(u') = \lh(v')$ and $u'\sqsubseteq u$, $v' \sqsubseteq v$, then $(u',v') \in T$. We will frequently identify a pair $(u,v)$ of finite sequences of the same length $n \in \om$ with the sequence of pairs $((u_0,v_0),\dots,(u_{n-1},v_{n-1}))$. The \emph{body of a tree of pairs $T$} is $[T] = \set{(\alpha,\gamma) \in \baire \times \baire}{(\forall n)[(\alpha \upharpoonright n, \gamma \upharpoonright n) \in T]}$. We will also consider trees of triples with the analogous notions.

We will often deal with partial functions on a space \ca{X} to a space \ca{Y}. These will be denoted as $f: \ca{X} \rightharpoonup \ca{Y}$. By $f(x) \downarrow$ we mean that $f$ is defined on $x$.

\subsection{Effective notions}

We will assume that the reader is familiar with the topic of effective descriptive theory. The usual reference to the latter is \cite{yiannis_dst}. Let us recall that a sequence $(x_n)_{\n}$ in the complete separable metric space $(\ca{X},d_{\ca{X}})$ is a \emph{recursive presentation} of $(\ca{X},d_{\ca{X}})$ if it forms a dense subset of $(\ca{X},d_{\ca{X}})$ and the relations $P, Q \subseteq \om^3$, defined by
\begin{align*}
P(i,j,k) \iff& d(x_i,x_j) < \pq{k}\\
Q(i,j,k) \iff& d(x_i,x_j) \leq \pq{k}
\end{align*}
are recursive. The metric space $(\ca{X},d_{\ca{X}})$ is \emph{recursively presented} if it admits a recursive presentation. 

When saying that \ca{X} is \emph{recursive Polish} we mean that the latter space is in fact given together with a suitable metric and a recursive presentation. Actually we go further by fixing for every recursive Polish space \ca{X} a metric $d_\ca{X}$ and a sequence $(\dense{\ca{X}}{i})_{i \in \om}$, which is a recursive presentation of $(\ca{X},d_{\ca{X}})$.

For every recursive Polish space \ca{X} we fix the numbering $(N(\ca{X},s))_{s \in \om}$ of a neighborhood basis for \ca{X},
\begin{align*}
N(\ca{X},s) 
=& \ B_{d_\ca{X}}(\dense{\ca{X}}{(s)_0},\pq{(s)_1})\\
=& \ \text{the $d_\ca{X}$-open ball of \ca{X} with center $\dense{\ca{X}}{(s)_0}$ and radius $\pq{(s)_1}$.}
\end{align*}

In this article we focus on the spaces $\cantor$, $\baire$, $[a,b]$, $\R$, and their finite products.  Recursive presentations are preserved by finite products the natural way. It is convenient in the sequel to adopt the the maximum metric for finite products. Given recursive Polish spaces $\ca{X}_0, \dots, \ca{X}_n$ and $\ca{X} = \prod_{i=0}^n \ca{X}_i$ we put
\begin{align*}
d_{\ca{X}}(\vec{x},\vec{y}) =& \ \max\set{d_{\ca{X}_i}(x_i,y_i)}{i=0,1,\dots, n}\\
\dense{\ca{X}}{s} =& \ (\dense{\ca{X}_0}{(s)_0}, \dots, \dense{\ca{X}_n}{(s)_n}),
\end{align*}
so that a neighborhood of the product is the product of neighborhoods; in fact it is easy to check that
\[
N(\ca{X}_0 \times \dots \times \ca{X}_n,s) = N(\ca{X}_0, \langle ((s)_0)_0, (s)_1 \rangle) \times \dots \times N(\ca{X}_n, \langle ((s)_n)_0, (s)_1 \rangle).
\]
Given a finite sequence $(t_0,\dots, t_n)$ of natural numbers the product $N(\ca{X}_0,t_0) \times \dots \times N(\ca{X}_n,t_n)$ is easily a semirecursive subset of $\ca{X} = \prod_{i=0}^n \ca{X}_i$, in fact this holds uniformly, \ie there is a recursive function $\fprod(\ca{X}_0,\dots,\ca{X}_n) \equiv \fprod: \om \to \baire$ such that
\begin{align}\label{equation product of neighborhoods}
N(\ca{X}_0,t_0) \times \dots \times N(\ca{X}_n,t_n) = \cup_{i \in \om} N(\ca{X}, \fprod(\langle t_0, \dots, t_n\rangle)(i))
\end{align}

In the sequel we consider the usual Kleene pointclasses $\Sigma^i_n$, $\Pi^i_n$, and $\Delta^i_n$, where $\n$ and $i=0,1$, together with their relativized versions.\smallskip

\textbf{Terminology.} Given a partial function $f: \ca{X} \rightharpoonup \ca{Y}$ between recursive Polish spaces, and sets $A \subseteq \text{Domain}(f)$, $R^f \subseteq \ca{X} \times \om$ we say that \emph{$f$ is computed by $R^f$ on $A$} if for all $x \in A$ and all $s \in \om$ we have that
\[
f(x) \in N(\ca{Y},s) \iff R^f(x,s).
\]
Partial $\pointcl$-recursive functions are exactly the ones, which are computed on their domain by a set in $\pointcl$. 

\subsubsection*{Universal systems} Given a class of sets $\pointcl$ in Polish spaces we denote by $\pointcl \upharpoonright \ca{X}$ the family of all subsets of \ca{X}, which belong in $\pointcl$. Recall that a set $G \subseteq \baire \times \ca{X}$ is \emph{universal} for $\pointcl \upharpoonright \ca{X}$ if $G$ is in $\pointcl$ and for all $P \subseteq \ca{X}$ we have that
\[
P \in \pointcl \iff \ \text{there is some $\alpha \in \baire$ such that} \ G(\alpha,x).
\]
As it is well-known the pointclasses $\bolds^i_n$, $i=0,1$ admit a universal system in a natural way. We fix the symbols $\ocode$ and $\fcode$ for the cases $\pointcl = \bolds^0_1$ and $\pointcl = \boldp^0_1$ respectively:

For every Polish space \ca{X} we define the set $\ocode^\ca{X}, \fcode^\ca{X} \subseteq \baire \times \ca{X}$ by
\begin{align*}
\ocode^\ca{X}(\alpha,x) \iff& (\exists n)[x \in N(\ca{X},\alpha(n))]\;,\\
\fcode^\ca{X}(\alpha,x) \iff& \neg \ocode^\ca{X}(\alpha,x).
\end{align*}
It is clear that the sets $\ocode^\ca{X}, \fcode^\ca{X}$ are universal for $\tboldsymbol{\Sigma}^0_1 \upharpoonright \ca{X}$ and $\tboldsymbol{\Pi}^0_1 \upharpoonright \ca{X}$ respectively. By a simple induction one can find universal sets for all pointclasses $\bolds^0_n$.

We also consider the following universal set for $\bolds^1_1$:
\begin{align*}
\scode^\ca{X}_1(\alpha,x) \iff& (\exists \gamma) \fcode^{\ca{X} \times \baire}(\alpha,x,\gamma).
\end{align*}
We fix the preceding universal sets throughout the rest of this article.

By saying that $\alpha$ is an \emph{open code} for $P\subseteq \ca{X}$ we mean that $P$ is the $\alpha$-section of the set $\ocode^\ca{X}$. The similar meaning applies \emph{closed} and \emph{analytic} codes.

One key aspect of our chosen universal system systems is that they are \emph{good}, \ie for every space \ca{X} of the form $\om^n \times \baire^m$, where $n,m \geq 0$, and every Polish space $\ca{Y}$ there is a recursive function $S: \baire \times \ca{X} \to \baire$ such that for all $\ep$, $\alpha$, $y$ it holds
\[
\gcode^{\baire \times \ca{X}}(\ep,\alpha,y) \iff \gcode^{\ca{X}}(S(\ep,\alpha),y),
\]
where $\pointcl$ is the pointclass under discussion and $\gcode^\ca{X}$ is the corresponding universal set. This is verified by straight-forward computations, which we omit.

Another important fact about our universal sets, is that the class of their recursive sections parametrize the corresponding effective pointclass. For example a subset of a recursive Polish space $\ca{X}$ is $\sig$ exactly when it is the $\alpha$-section of $\scode^{\ca{X}}_1$ for some recursive $\alpha$.\smallskip

\textbf{More on notation.} We will usually write $\scode^{\ca{X}}_1(\alpha)$ for the $\alpha$-section of $\scode^{\ca{X}}_1$, and do similarly for the sections of all other sets.

\subsection{Borel codes}

\label{subsection Borel codes}

With the help of the preceding universal sets one can encode the partial $\alpha$-recursive functions from a recursive Polish space to another, and similarly for the $\del(\alpha)$-recursive ones, see \cite[7A]{yiannis_dst}. We denote by $\rfn{\alpha}^{\ca{X},\ca{Y}}_{\pointcl}$ the largest partial function $f: \ca{X} \rightharpoonup \ca{Y}$, which is computed on its domain by $\gcode^{\ca{X} \times \baire}(\alpha)$, where $\gcode$ is as above.

If $\pointcl = \Sigma^i_n$, where $i=0,1$ and $\n$, then a partial function $f: \ca{X} \rightharpoonup \ca{Y}$ is $\pointcl$-recursive on its domain if and only if there is some recursive $\alpha \in \baire$ such that for all $x$ with $f(x) \downarrow$ we have $f(x) = \rfn{\alpha}^{\ca{X},\ca{Y}}_{\pointcl}$\;; this is essentially a result of Kleene, see \cite[7A.1]{yiannis_dst}.

If we omit $\pointcl$ we always mean that $\pointcl = \Sigma^0_1$, so the function $\rfn{\alpha}^{\ca{X},\ca{Y}}$ is $\alpha$-recursive. If we simply write $\rfn{\alpha}$ we mean that $\ca{X} = \om$ and $\ca{Y} = \baire$, \ie $\rfn{\alpha} = \rfn{\alpha}^{\om,\baire}_{\Sigma^0_1}$.

The set $\bcodefam$ of \emph{Borel codes} is defined by transfinite recursion
\begin{align*}
\alpha \in \bcodefam_0 \iff& \ \alpha(0) = 0\;,\\
\alpha \in \bcodefam_\xi \iff& \ \alpha(0) = 1 \ \& \ (\forall n)[\rfn{\alpha^\ast}(n) \downarrow \ \& \ (\exists \eta < \xi)\rfn{\alpha^\ast}(n) \in \bcodefam_\eta]\;,\\
\end{align*}
for $\xi > 0$ and
\[
\bcodefam = \ \cup_{\xi < \om_1} \bcodefam_\xi\;.
\]
It is clear that $\bcodefam_\eta \subseteq \bcodefam_\xi$ for all $1 \leq \eta < \xi$. Moreover as it is well-known the set $\bcodefam$ is $\pii$ and not Borel.

For $\alpha \in \bcodefam$ we put $\normb{\alpha} = \ \text{the least $\xi < \om_1$ such that} \ \alpha \in \bcodefam_\xi$.

With the help of $(\bcodefam_\xi)_{\xi < \om_1}$ we can encode the families $(\bolds^0_\xi \upharpoonright \ca{X})_{\xi < \om_1}$ for every Polish space \ca{X}. More specifically, given a Polish space $\ca{X}$ with a countable basis $\bolds^0_0 \upharpoonright \ca{X} = \set{V_s}{s \in \om}$,\footnote{If \ca{X} is recursive Polish we will always assume that $V_s = N(\ca{X},s)$, $s \in \om$.} define by recursion
\begin{align*}
\bcf^{\ca{X}}_0:& \ \bcodefam_0 \surj \set{P}{\ca{X} \setminus P \in \bolds^0_0 \upharpoonright \ca{X}}: \alpha \mapsto \ca{X} \setminus V_{\alpha(1)}\;,\\
\bcf^{\ca{X}}_\xi:& \ \bcodefam_\xi \surj \bolds^0_\xi \upharpoonright \ca{X}: \alpha \mapsto \cup_{n \in \om} \left(\ca{X} \setminus \bcf^{\ca{X}}_{\normb{\rfn{\alpha^\ast}(n)}}(\rfn{\alpha^\ast}(n))\right).
\end{align*}
By an easy induction we can see that each $\bcf^{\ca{X}}_\xi \equiv \bcf_\xi$ is indeed surjective and that $\bcf_\xi \upharpoonright \bcodefam_\eta = \bcf_\eta$ for all $1 \leq \eta < \xi$.

Finally we define $\bcf^{\ca{X}} \to \text{Borel subsets of} \ \ca{X}: \alpha \mapsto \bcf^{\ca{X}}(\alpha)$. In other words every $\alpha \in \bcodefam$ is a \emph{Borel code} for the set $\bcf^{\ca{X}}(\alpha)$.

A subset $A$ of a recursive Polish space \ca{X} is \emph{lightface $\Sigma^0_\xi$} if we have $A = \bcf^\ca{X}(\alpha)$ for some \emph{recursive} $\alpha \in \bcodefam$ with $\normb{\alpha} = \xi$. It is a well-known consequence of the Souslin-Kleene Theorem that
\[
\del = \cup_{\xi < \ck} \; \Sigma^0_\xi = \cup_{\xi < \om_1} \; \Sigma^0_\xi.
\]
The union $\cup_{\xi < \om_1} \; \Sigma^0_\xi$ is also denoted as \HYP \ (the class of \emph{hyperarithmetical sets}), and the preceding equality extends the well-known result of Kleene $\del \upharpoonright \om = \HYP \upharpoonright \om$. 

\section{The uniform Dyck separation}

\label{section the uniform Dyck separation}

We first remark that not all monotone sets are Borel. Recall that a set $A \subseteq \ca{X}$ is \emph{$\bolds^1_1$-complete} if for every zero-dimensional Polish $\ca{Z}$ and every $\bolds^1_1$ set $P \subseteq \ca{Z}$ there is a continuous function $f: \ca{Z} \to \ca{X}$, which \emph{reduces} $P$ to $A$, \ie $P = f^{-1}[A]$. It is clear that no $\bolds^1_1$-complete set is Borel.

\begin{proposition}
\label{proposition a sig monotone set which is not Borel}
There is a $\sig$ monotone subset of $\cantor$ which is $\bolds^1_1$-complete.
\end{proposition} 

\begin{proof}
We consider the sets $\LO$ and $\WO$ of codes of countable linear and well-orderings respectively from the Introduction. It is not hard to see that the set $\LO$ is arithmetical and that the set $\WO$ is $\pii$.

We define
\[
A := \set{\beta \in \cantor}{(\exists \alpha \in \LO)[\alpha \not \in \WO \ \& \ \alpha \subseteq \beta]}.
\]
Evidently $A$ is a monotone $\sig$ set. To show that it is $\bolds^1_1$-complete we consider the space $\Tr$ of all non-empty trees on the naturals (this is isomorphic to a closed subset of $\cantor$), and the set $\IF$ of all \emph{ill-founded} trees, \ie the trees with an infinite branch. As it is well-known it suffices to define a function $f : \Tr \to \cantor$ such that $\IF = f^{-1}[A]$.

Moreover we consider the \emph{Kleene-Brouwer}, also  known as \emph{Luzin-Sierpinski} linear ordering $\kbleq$ on the set of all finite sequences of natural numbers. As it is well-known a tree $T$ is ill-founded exactly when $T$ with the restriction $\kbleq \upharpoonright T$ of $\kbleq$ on $T$ is not a well-ordering. We can encode $\kbleq \upharpoonright T$ by some $\alpha(T) \in \LO$ the natural way, \ie $\alpha(T)$ is the unique member of $\cantor$ which satisfies
\begin{align*}
u = (u_0,\dots,u_{n-1}) \in T \iff& \ \langle u_0,\dots,u_{n-1} \rangle \in \text{Field}(\alpha(T)),\\
&\hspace*{-55mm}\text{for} \ u = (u_0,\dots,u_{n-1}), v = (v_0,\dots,v_{m-1}) \in T \ \text{we have}\\  
u \kbleq v \longleftrightarrow& \ \langle u_0,\dots, u_{n-1} \rangle \leq_{\alpha(T)} \langle v_0,\dots, v_{m-1}\rangle.
\end{align*}
The function $T \in \Tr \mapsto \alpha(T) \in \LO$ is easily recursive. We show that it reduces $\IF$ to $A$. One of the inclusions is obvious: if $T \in \IF$ then $\alpha(T)$ is in $\LO \setminus \WO$, and so in particular $\alpha(T) \in A$. 

To prove the converse we remark first that for all $\alpha, \beta \in \LO$ with $\alpha \subseteq \beta$ and all $n,m$ with $n <_{\alpha} m$ we have $n <_{\beta} m$. This is fairly easy to see. Let $\alpha$, $\beta$, $n,m$ be as just described; in particular we have $\alpha (\langle n,n \rangle) = \alpha (\langle m,m \rangle) = \alpha (\langle n,m \rangle) = 1$, and using that $\alpha \subseteq \beta$ it follows $\beta (\langle n, n\rangle ) = \beta (\langle m, m\rangle ) = \beta (\langle n, m\rangle ) =1$, \ie $n \leq_{\beta} m$. If we had $m \leq_{\beta} n$ it would follow $n = m$ because $\beta \in \LO$. Since $\alpha \in \LO$ and $n = m \in \text{Field}(\alpha)$ we would also have $m \leq_{\alpha} n$, a contradiction.

Now if $\alpha(T) \in A$ and $\alpha \in \LO \setminus \WO$ is such that $\alpha \subseteq \alpha(T)$, using the preceding remark we have that every strictly decreasing sequence under $\leq_{\alpha}$ is also strictly decreasing under $\leq_{\alpha(T)}$. Hence $\alpha(T) \not \in \WO$, \ie $T \in \IF$.
\end{proof}\bigskip

Next we \emph{prove the first part of Theorem \ref{theorem uniformity version Dyck}}.\smallskip

We essentially need to give a constructive proof to Theorem \ref{theorem Dyck} and use Borel codes. To do this we adjust the constructive proof  of the Luzin Separation Theorem (see \cite[2E.1]{yiannis_dst}) to our setting.

Let $A, B$ be non-empty disjoint analytic subsets of $\cantor$, and let $T$ and $S$ be trees of pairs such that
\begin{align*}
x \in A \iff& \ (\exists \gamma)(\forall t)[(x\upharpoonright t, \gamma \upharpoonright t) \in T]\\
y \in B \iff& \ (\exists \delta)(\forall t)[(y\upharpoonright t, \delta \upharpoonright t) \in S].
\end{align*}
We define the tree $J$ of quadruples by
\[
(u,c,v,d) \in J \iff (u,c) \in T \ \& \ (v,d) \in S \ \& \ (\forall i < n)[u(i) = 1 \ \longrightarrow v(i)=1],
\]
where $u, v \in 2^{< \om}$, $c,d \in \om^{<\om}$, and $u,v,c,d$ have the same length $n$. An infinite branch $(x,\gamma,y,\delta)$ in $[J]$ would imply that $(x,\gamma) \in [T]$, $(y,\delta) \in [S]$ and $x \subseteq y$. Then we would have that $x \in A$, $y \in B$, and since $A$ is monotone $y$ would be also in $A$. This would imply that $A \cap B \neq \emptyset$, which contradicts our hypothesis. Hence the tree $J$ is well-founded.\footnote{Notice that if we strengthen the last condition in the definition of $J$ to ``$u=v$", then what we get is essentially the well-founded tree of triples of the proof of the Luzin Separation Theorem as it is given in \cite{yiannis_dst}.  The key idea in our proof is to view the latter tree as a special case of a tree defined by the formula $(u,c) \in T \ \& \ (v,d) \in S \ \& \ P(u,c,v,d)$, where $P$ is a property that reflects the additional hypothesis on $A$.}

We will define by bar recursion on the branches of $J$ a family $(C_\sigma)_{\sigma \in J}$ of subsets of $\cantor$ such that for all $\sigma=(u,c,v,d) \in J$,\smallskip

(a) $C_\sigma$ is semi-positive,\smallskip

(b) $C_\sigma$ separates $\pr[T_{(u,c)}]$ from $\pr[S_{(v,d)}]$.

From this it follows that $C:=C_\emptyset$ is Borel semi-positive which separates $A =  \pr[T_\emptyset]$ from $B =  \pr[S_\emptyset]$.

Now we proceed to the definition of the sets $C_\sigma$. Let $\sigma = (u,c,v,d) \in J$, and assume that $C_{\sigma'}$ has been defined and satisfies (a) and (b) above for all $\sigma' \in J$, which extend $\sigma$ properly. Clearly we have that
\begin{align*}
\pr[T_{(u,c)}] =& \ \cup_{(t,n) \in 2 \times \om} \ \pr[T_{\cn{(u,c)}{(t,n)}}], \quad \text{and}\\
\pr[S_{(v,d)}] =& \ \cup_{(s,m) \in 2\times\om} \ \pr[S_{\cn{(v,d)}{(s,m)}}].
\end{align*}
We will define a family $(D^\sigma_{(t,n,s,m)})_{t,n,s,m}$ of semi-positive sets such that for all $(t,n,s,m)$ the set $D^\sigma_{(t,n,s,m)}$ separates $\pr[T_{\cn{(u,c)}{(t,n)}}]$ from $\pr[S_{\cn{(v,d)}{(s,m)}}]$. Then it is easy to see that the set
\[
C_\sigma : = \cup_{(t,n) \in 2 \times \om} \cap_{(s,m) \in 2\times\om} D^\sigma_{(t,n,s,m)}
\]
separates $\pr[T_{(u,c)}]$ from $\pr[S_{(v,d)}]$. Moreover the preceding $C_\sigma$ is clearly semi-positive.

Now we proceed to the definition of $(D^\sigma_{(t,n,s,m)})_{t,n,s,m}$. Suppose that  $t,n,s,m$ are given naturals with $t,s \leq 1$.

\emph{Case 1}. $\cn{\sigma}{(s,n,t,m)} \not \in J$. We consider the following sub-cases. 

If $\cn{(u,c)}{(t,n)} \not \in T$ then $\pr[T_{\cn{(u,c)}{(t,n)}}] = \emptyset$, so the set $D^\sigma_{(s,n,t,m)} = \emptyset$ separates $\pr[T_{\cn{(u,c)}{(t,n)}}]$ from $\pr[S_{\cn{(v,d)}{(s,m)}}]$ and is semi-positive. 

If $\cn{(u,c)}{(t,n)} \in T$ and $ \cn{(v,d)}{(s,m)} \not \in S$, then $\pr[S_{\cn{(v,d)}{(s,m)}}] = \emptyset$, so the set $D^\sigma_{(s,n,t,m)} = \cantor$ separates $\pr[T_{\cn{(u,c)}{(t,n)}}]$ from $\pr[S_{\cn{(v,d)}{(s,m)}}]$ and is semi-positive. 

The remaining sub-case is when $\cn{(u,c)}{(t,n)} \in T$ and $\cn{(v,d)}{(s,m)} \in S$. Since $\cn{\sigma}{(t,n,s,m)} = \cn{(u,c,v,d)}{(t,n,s,m)} \not \in J$ and $\sigma \in J$, it follows that $t=1$ and $s = 0$. Hence we take the positive set $D^\sigma_{(t,n,s,m)} = \set{x \in \cantor}{x(\lh(u)) = 1}$. Then the latter separates $\pr[T_{\cn{(u,c)}{(1,n)}}]$ from $\pr[S_{\cn{(v,d)}{(0,m)}}]$.

\emph{Case 2.} $\cn{\sigma}{(s,n,t,m)} \in J$. In this case we take $D^\sigma_{(t,n,s,m)}$ to be $C_{\cn{\sigma}{(t,n,s,m)}}$, which by the inductive hypothesis is defined, is semi-positive and separates $\pr[T_{\cn{(u,c)}{(t,n)}}]$ from $\pr[S_{\cn{(v,d)}{(s,m)}}]$. 

This completes the inductive step, and so we have a proof for Theorem \ref{theorem Dyck}. 

The construction of the function $u$ proceeds as in the proof of the Suslin-Kleene Theorem \cite[7B.3, 7B.4]{yiannis_dst} using Kleene's Recursion Theorem \cite[7A.2]{yiannis_dst}. (Notice that the cases in our proof are defined by $(T,S)$-recursive conditions.) 

To explain this better, using the preceding analysis we can find a partial recursive function $d(\ep,\zeta,\eta,i,t,n,s,m)$, where $\zeta, \eta \in \cantor$ such that whenever $\scode^{\cantor}_1(\alpha)$ is monotone and disjoint from $\scode^{\cantor}_1(\beta)$, $\zeta_T, \zeta_S \in \cantor$ encode the trees $T$ and $S$ which correspond to the analytic sets $A = \scode^{\cantor}_1(\alpha)$ and $B = \scode^{\cantor}_1(\beta)$, and $i \in \om$ encodes a tuple $\sigma$ with $\sigma \in J$, then $d(\ep,\zeta_T,\zeta_S,i,t,n,s,m)$ is defined and gives a \emph{Borel code} for $D^\sigma_{(t,n,s,m)}$ according to the preceding case distinction.  \footnote{By ``$\zeta_T$ encodes $T$" we mean that $T = \set{\dec{n}}{\zeta_T(n)}=1$. Also the term ``$i$ encodes the tuple $\sigma = (u,c,v,d)$" means that $(i)_0, (i)_1, (i)_2, (i)_3 \in \Seq$ and that the sequences $\dec{(i)_k}$, $k= 0,1,2,3$ have the same length.}

Then we proceed exactly as in the proof of \cite[7B.3]{yiannis_dst} to obtain a function $v(\zeta,\eta)$ such that whenever $\zeta_T$ and $\zeta_S$ encode trees $T$ and $S$ as above then $v(\zeta_T,\zeta_S)$ is a Borel code for a separating semi-positive set $C$. We omit the details.

It remains to show that the preceding $\zeta_T$ can be obtained as a recursive function of $\alpha$. (See the proof of \cite[7B.4]{yiannis_dst} - the $\zeta_1$ is $\mathbf{u}_1(\alpha)$, which by a typo is also denoted also by $\theta_1(\alpha)$.)

This is fairly easy to do; we consider the universal set $\scode^{\cantor}_1 \subseteq \cantor \times \baire$ and a recursive tree $\tilde{T}$ of triples such that
\[
\scode^{\baire \times\cantor}_1(\alpha,x) \iff (\exists \gamma)(\forall t)[(\alpha \upharpoonright t, x \upharpoonright t, \gamma \upharpoonright t) \in \tilde{T}].
\]
Then $\zeta_T \equiv \zeta(\alpha) \in \cantor$ is given by
\begin{align*}
\zeta(\alpha)(n) = 1 
\iff& \ (n)_0, (n)_1 \in \Seq \ \& \ \lh(\dec{(n)_0}) = \lh(\dec{(n)_1})\\
& \ \& \ (\alpha \upharpoonright \lh(\dec{(n)_0}), \dec{(n)_0}, \dec{(n)_1} \in \tilde{T}).
\end{align*}
This shows how to obtain the function $u$ as in the statement of Theorem \ref{theorem uniformity version Dyck}.

\subsection{Semi-positive codes.}

\label{subsection semi-positive codes}

We introduce the following hierarchy of the family $\spclass$ of all semi-positive subsets of $\cantor$,
\begin{align*}
V_0 =& \ \emptyset, \quad V_1 = \cantor, \quad V_{n+2}:=\posb_n = \set{x \in \cantor}{x(n) =1 }, \quad \n;\\ 
\spclass_0 =& \ \set{V_n}{\n};\\
\spclass_\xi =& \ \set{\cup_{i \in \om}\cap_{j \in \om}A_{ij}}{\text{for all $i,j$ there is $\xi_{ij} < \xi$ such that} \ A_{ij} \in \spclass_{\xi_{ij}}},\\
& \ \text{where} \ \hspace*{0mm} 1 \leq \xi < \om_1.
\end{align*}
It is not difficult to verify that $\spclass_\eta \subseteq \spclass_\xi$ for all countable $\eta < \xi$, and that
\[
\spclass = \text{the family of all semi-positive sets} = \cup_{\xi < \om_1} \spclass_\xi.
\]
We also define the family $(\spcode_\xi)_{\xi < \om_1}$ of codes for $\spclass$,
\begin{align*}
\alpha \in \spcode_0 \iff& \ \alpha(0) = 0\\
\alpha \in \spcode_\xi \iff& \ \alpha(0) = 1 \ \& \ (\forall i,j)(\exists \eta < \xi)[\rfn{\alpha^\ast}(\langle i,j \rangle) \in \spcode_\eta],
\end{align*}
for all $\xi > 0$ and all $\alpha \in \baire$, and
\[
\spcode = \cup_{\xi < \om_1}\spcode_\xi.
\]
The members of $\spcode$ will be called \emph{semi-positive codes}. It is immediate that $\spcode_\eta \subseteq \spcode_\xi$ for all $1 \leq \eta < \xi$.

Notice the the set of all semi-positive codes is essentially the same as the set $\bcodefam$ of all Borel codes, the only difference arising from the fact that the function $(i,j) \mapsto \langle i,j \rangle$ is not surjective (for example it does not obtain the value $\langle 0,0,0\rangle$ and $\rfn{\alpha^\ast}(\langle 0,0,0 \rangle)$ can be an arbitrary member of $\baire$ even if $\alpha \in \spcode$). However this difference is a minor point.\footnote{Actually we could replace the function $(i,j) \mapsto \langle i,j \rangle$ with some recursive bijection between $\om^2$ to $\om$. Under this modification the sets $\bcodefam$, $\spcode$ would be exactly the same. We find it though counter-intuitive to use $\bcodefam$ in the place of $\spcode$, and thus we keep the definition above.} It is easy to see that the sets $\bcodefam$ and $\spcode$ are recursively bi-reducible, \ie there are recursive functions $f, g: \baire \to \baire$ such that $\bcodefam = f^{-1}[\spcode]$ and $\spcode = g^{-1}[\bcodefam]$. As a consequence we have that the set $\spcode$ is $\pii$ and not Borel. (Of course one can also prove these facts directly.)

Given $\alpha \in \spcode$ we put
\[
\normsp{\alpha} = \text{the least $\xi < \om_1$ such that} \ \alpha \in \spcode_\xi.
\]
It is evident that $\normsp{\rfn{\alpha^\ast}(\langle i,j \rangle)} < \normsp{\alpha}$ for $\alpha \in \cup_{1 \leq \xi < \om_1} \spcode_\xi$. Moreover it holds
\[
\normsp{\alpha} = \sup_{i,j \in \om} \left(\normsp{\rfn{\alpha^\ast}(\langle i,j \rangle)} + 1 \right)
\]
for all $\alpha \in \spcode$ with $\normsp{\alpha} > 0$.

The coding $\spcfun{\xi}$ of the family $\spclass_\xi$ is defined by recursion on $\xi$,
\begin{align*}
\spcfun{0} : \spcode_0 \surj \spclass_0:& \ \spcfun{0}(\alpha) = V_{\alpha^\ast(1)}\\
\spcfun{\xi} : \spcode_\xi \surj \spclass_\xi:& \ \spcfun{\xi}(\alpha) = \cup_i\cap_j \spcfun{\normsp{\rfn{\alpha^\ast}(\langle i,j \rangle)}}(\rfn{\alpha^\ast}(\langle i,j \rangle)).
\end{align*}
It is again easy to verify that each $\spcfun{\xi}$ is indeed surjective, and that $\spcfun{\xi} \upharpoonright \spcode_\eta = \spcfun{\eta}$ for all countable $\eta < \xi$. Thus the function
\[
\spcf : = \cup_{\xi < \om_1}\spcfun{\xi} : \spcode \surj \spclass
\]
defines a coding of the family $\spclass$. We say that an $\alpha \in \spcode$ is a \emph{code for the semi-positive} set $A \subseteq \cantor$ if $A = \spcf(\alpha)$.\bigskip

Now we are ready to \emph{prove the second part of Theorem \ref{theorem uniformity version Dyck}}.\smallskip

The difference from the first part is that here we choose the codes for the separating sets $D^\sigma_{(t,n,s,m)}$ with respect to $\spcf$. We give the core of the argument.

First we fix the following codes for the members of $\spclass_0$,
\[
\beta_k =
\begin{cases}
(0,0,\dots), & \ k=0,\\
\cn{(0,1)}{(0,0,\dots)}, & \ k = 1,\\
\cn{(0,k-2)}{(0,0,\dots)}, & \ k > 1.
\end{cases}
\]
If $D^\sigma_{(s,n,t,m)}$ is chosen from $\spclass_0$ we take the coding provided by the preceding $\beta_k$'s. In order to define a code of 
\[
C_\sigma:= \cup_{(t,n) \in 2 \times \om} \cap_{(s,m) \in 2\times\om} D^\sigma_{(t,n,s,m)}
\]
out of the codes $\rfn{\alpha}(\langle \langle t,n \rangle, \langle s,m \rangle \rangle)$ of $D^\sigma_{(t,n,s,m)}$ we do the following procedure. From Kleene's Recursion Theorem \cite[7A.2]{yiannis_dst} we can find a partial recursive function $\rho: \baire \rightharpoonup \baire$ such that for all $\alpha \in \baire$ for which $\rfn{\alpha}(\langle \langle t,n \rangle, \langle s,m \rangle \rangle) \downarrow$ for all $t,n,s,m$ it holds $\rho(\alpha) \downarrow$ and:

\begin{align*}
\rho(\alpha)(0) =& \ 1;\\
\rfn{\rho(\alpha)^\ast}(\langle \langle t,n \rangle, \langle s,m \rangle \rangle) =& \
\begin{cases}
\rfn{\alpha}(\langle \langle t,n \rangle, \langle s,m \rangle \rangle), & \ (t,n), (s,m) \in \om \times 2,\\
(k \mapsto 1), & \ (t,n) \in \om \times 2, \ s \in \om, \ m \geq 2,\\
(k \mapsto 0), & \ t,s,m \in \om, \ n \geq 2;
\end{cases}\\
\rfn{\rho(\alpha)^\ast}(k) =& \ 0  \quad \text{in all other cases}.
\end{align*}
It is then clear that for all $\alpha \in \spcode$ it holds
\begin{align*}
\spcf(\rho(\alpha)) 
=& \ \cup_{(t,n) \in 2 \times \om} \cap_{(s,m) \in 2\times\om} \ \spcf(\rfn{\rho(\alpha)^\ast}(\langle \langle t,n \rangle, \langle s,m \rangle \rangle))\\
=& \ \cup_{(t,n) \in 2 \times \om} \cap_{(s,m) \in 2\times\om} \ \spcf(\rfn{\alpha}(\langle \langle t,n \rangle, \langle s,m \rangle \rangle))\\
=& \ \cup_{(t,n) \in 2 \times \om} \cap_{(s,m) \in 2\times\om} \ D^\sigma_{(t,n,s,m)}\\
=& \ C_\sigma.
\end{align*}
Thus the function $\rho$ delivers a recursive way for obtaining a code for $C_\sigma$ out of recursively-given codes for $D^\sigma_{(t,n,s,m)}\;$.

\subsection{\HYP-semi-positive is exactly \del \ and semi-positive.}

First we give the \HYP \ version of semi-positive sets.

\begin{definition}\normalfont
\label{definition effective semipositive}
Let $\ep \in \baire$. The family $\speffclass_\xi(\ep)$ of all \emph{$\HYP(\ep)$-semi-positive} sets of \emph{order $\xi$} is defined by
\[
\speffclass_\xi(\ep) = \set{\spcf(\alpha)}{\alpha \ \text{is $\ep$-recursive and} \ \normsp{\alpha} = \xi}.
\]
The family of all \emph{$\HYP(\ep)$-semi-positive sets} is
\[
\cup_{\xi} \ \speffclass_\xi(\ep).
\]
When $\ep$ is recursive we say ``\HYP-semi-positive'' instead of ``$\HYP(\ep)$-semi-positive'' and similarly we write $\speffclass$ instead of $\speffclass(\ep)$.

It is clear that a set is semi-positive exactly when it is $\HYP(\ep)$-semi-positive for some $\ep \in \baire$. 
\end{definition}

Corollary \ref{corollary effective Dyck} now reads that the $\del$ semi-positive sets are exactly the \HYP-semi-positive ones, (which are exactly the $\del$ and monotones).\smallskip

We prove some basic facts about the preceding notions in a series of Lemmata. Recall that a \emph{norm} on a set $P \subseteq \ca{X}$ is any function $\varphi: P \to$ Ordinals. Given a pointclass $\Gamma$ and $P \subseteq \ca{X}$ in $\Gamma$ we say that a norm $\varphi$ on $P$ is a \emph{$\Gamma$-norm} if there are relations $\leq_\Gamma, \leq_{\neg \Gamma} \subseteq \ca{X} \times \ca{X}$ in $\Gamma$ and $\neg \Gamma$ respectively such that for all $y \in P$ and all $x \in \ca{X}$ it holds
\[
[P(x) \ \& \ \varphi(x) \leq \varphi(y)] \iff x \leq_{\Gamma} y \iff x \leq_{\neg \Gamma} y.
\]
see \cite[4B]{yiannis_dst}. For example the function $\beta \in \WO \mapsto |\beta|$ is a $\pii$ norm on $\WO$, \cf \cite[4A]{yiannis_dst}. Here we have another example of a norm, namely $\alpha \in \spcode \mapsto \normsp{\alpha}$. It can be proved that it is also a $\pii$-norm, but here we want to do something different. More specifically we need to compare $\normsp{\alpha}$ with $|\beta|$ in a similar fashion as above. Although we need only one of all possible combinations, we give the full result, as we think that it is interesting in its own right.

\begin{lemma}
\label{lemma norm comparison}
There are $\sig$ relations $\leq_{\sig}$, $\preceq_{\sig}$ and $\pii$ relations $\leq_{\pii}$, $\preceq_{\pii}$ which satisfy for all $\alpha, \beta \in \baire$ the following properties.
\begin{enumerate}
\item[\tu{(}a\tu{)}] If $\beta \in \WO$ then
\[
[\alpha \in \spcode \ \& \ \normsp{\alpha} \leq |\beta|] \iff \alpha \leq_{\sig} \beta \iff \alpha \leq_{\pii} \beta.
\]
\item[\tu{(}b\tu{)}] If $\alpha \in \spcode$ then
\[
[\beta \in \WO \ \& \ |\beta| \leq \normsp{\alpha}] \iff \beta \preceq_{\sig} \alpha \iff \beta \preceq_{\pii} \alpha.
\]
\end{enumerate}
In particular the function $\alpha \in \spcode \mapsto \normsp{\alpha}$ is a $\pii$-norm on $\spcode$.
\end{lemma}

\begin{proof}
We consider the operator $\Phi: \ca{P}(\baire \times \baire) \to  \ca{P}(\baire \times \baire)$ defined by
\begin{align*}
\Phi(A) =& \ \set{(\alpha,\beta) \in \baire \times \baire}{\alpha(0) = 0 \ \vee\\
&\hspace*{20mm} \vee \ (\forall i,j)(\exists n)[\rfn{\alpha^\ast}(\langle i,j \rangle) \downarrow \ \& \ (\rfn{\alpha^\ast}(\langle i,j \rangle),\restrict{\beta}{n}) \in A]}. 
\end{align*}
It is clear that $\Phi$ is monotone and $\pii$ ($\sig$) on $\pii$ (resp. $\sig$). It is also not hard to verify that for every fixed point $\leq$ of $\Phi$ and every $(\alpha, \beta) \in \baire \times \WO$ it holds
\[
[\alpha \in \spcode \ \& \ \normsp{\alpha} \leq |\beta|] \iff \alpha \leq \beta.
\]
The left-to-right direction is proved by induction on $\normsp{\alpha}$ and the converse by induction on $|\beta|$. Then we take $\leq_{\pii}$ and $\leq_{\sig}$ to be the least and the largest (respectively) fixed point of $\Phi$. The former relation is $\pii$ by the Norm Induction Theorem \cite[7C.8]{yiannis_dst} and the latter is $\sig$ by its dual form.

The second assertion is proved similarly by considering the operator
\begin{align*}
&\hspace*{-12mm}\Psi: \ca{P}(\baire \times \baire) \to  \ca{P}(\baire \times \baire):\\
\Psi(A) =& \ \set{(\beta,\alpha) \in \baire \times \baire}{|\beta| = 0 \ \vee\\
&\hspace*{20mm} \vee \ (\forall n)(\exists i,j)[\rfn{\alpha^\ast}(\langle i,j \rangle) \downarrow \ \& \ (\restrict{\beta}{n},\rfn{\alpha^\ast}(\langle i,j \rangle)) \in A]}. 
\end{align*}
Finally we notice that for all $\alpha \in \spcode$ and all $\gamma \in \baire$ it holds
\begin{align*}
[\gamma \in \spcode \ \& \ \normsp{\gamma} \leq \normsp{\alpha}] 
\iff& \ (\exists \beta)[\beta \preceq_{\sig}\alpha \ \& \ \gamma \leq_{\sig} \beta]\\
\iff& \ (\forall \beta)[(\beta \preceq_{\sig}\alpha \ \& \ \alpha \leq_{\sig}\beta) \ \longrightarrow \ \gamma \leq_{\pii} \beta].
\end{align*}
The preceding equivalences show that $\normsp{\cdot}$ is a $\pii$-norm on $\spcode$. (Notice that by exchanging $\leq$ and $\preceq$ on the right-hand side of the last two equivalences we obtain another proof that $|\cdot|$ is $\pii$-norm on $\bcodefam$.) This finishes the proof.
\end{proof}

\begin{lemma}
\label{lemma spc norm is recursive ordinal}
For all $\alpha \in \spcode$ the norm $\normsp{\alpha}$ is an $\alpha$-recursive ordinal.
\end{lemma}

\begin{proof}
We prove the assertion by induction on $\normsp{\alpha}$. 

The claim is clear for $\normsp{\alpha} = 0$. Let $\alpha \in \spcode$ be such that $\normsp{\alpha} > 0$, so that $\normsp{\alpha}$ is the supremum of all ordinals of the form $\normsp{\rfn{\alpha^\ast}(\langle i,j \rangle)}+1$.

From our inductive hypothesis it follows that for all $i,j$ the ordinal $\normsp{\rfn{\alpha^\ast}(\langle i,j \rangle)}$ is $\alpha$-recursive. Hence for all $i,j \in \om$ there is some $\alpha$-recursive $\beta \in \WO$, whose order type is $\normsp{\rfn{\alpha^\ast}(\langle i,j \rangle)}$.

We define
\[
P(i,j,\beta) \iff \beta \in \WO \ \& \ \rfn{\alpha^\ast}(\langle i,j \rangle) \leq_{\pii} \beta, 
\]
where $\leq_{\pii}$ is as in Lemma \ref{lemma norm comparison}. From the preceding we have that for all $i,j \in \om$ there is some $\alpha$-recursive $\beta$ such that $P(i,j,\beta)$. Moreover the set $P$ is clearly $\pii$. It follows from the $\Delta$-Selection Principle \cf \cite[4D.6]{yiannis_dst} that there is some $\del(\alpha)$-recursive function $f: \om^2 \to \baire$ such that for all $i,j \in \om$ it holds $P(i,j,f(i,j))$, \ie $f(i,j) \in \WO$ and
\[
\normsp{\rfn{\alpha^\ast}(\langle i,j \rangle)} \leq |f(i,j)|
\]
for all $i,j \in \om$.

It is easy to find a recursive function $g: \LO \to \LO$ such that whenever $\beta \in \WO$ it holds $g(\beta) \in \WO$ and $|g(\beta)| = |\beta| + 1$. Summing up:
\begin{align*}
\normsp{\alpha} 
=& \ \sup_{i,j \in \om} \left(\normsp{\rfn{\alpha^\ast}(\langle i,j \rangle)} + 1 \right)
\leq \ \sup_{i,j \in \om} |g(f(i,j))| < \ckr{\alpha},
\end{align*}
where in the last inequality we used that: (i) the function $g \circ f$ is $\del(\alpha)$-recursive and so the supremum $\sup_{i,j \in \om} |g(f(i,j))|$ is a $\del(\alpha)$-recursive ordinal; (ii) the supremum of all $\del(\alpha)$-recursive ordinals is $\ckr{\alpha}$ (Spector \cite{spector_recursive_well-orderings}); the latter supremum is not a maximum. Hence $\normsp{\alpha}$ is an $\alpha$-recursive ordinal.
\end{proof}

\begin{corollary}
\label{corollary union spc stops at the church-kleene ordinal}
It holds
\[
\cup_{\xi} \ \speffclass_\xi = \cup_{\xi < \ck} \ \speffclass_\xi.
\]
\end{corollary}

\begin{proof}
For the non-trivial direction, suppose that $A \in \speffclass_\xi$ for some countable ordinal $\xi$, and let $\alpha$ be recursive such that $A = \spcf(\alpha)$ and $\normsp{\alpha} = \xi$. Then from Lemma \ref{lemma spc norm is recursive ordinal} the ordinal $\normsp{\alpha}$ is recursive, \ie $\xi = \normsp{\alpha} < \ck$.
\end{proof}

\begin{lemma}
\label{lemma basic properties of effective semipositive B}
It holds
\[
\cup_{\xi < \ck} \ \speffclass_\xi  \subseteq \del.
\]
In fact the preceding inclusion holds recursively uniformly, \ie there is a recursive function $u = (u_1,u_2): \baire \to \baire \times \baire$ such that for all $\alpha \in \spcode$ the points $u_1(\alpha)$ and $u_2(\alpha)$ are analytic codes for $\spcf(\alpha)$ and $\cantor \setminus \spcf(\alpha)$ respectively.  
\end{lemma}

\begin{proof}
Same as the proof of \cite[7B.5]{yiannis_dst}.
\end{proof}\bigskip

With the help of the preceding results we can \emph{prove Corollary \ref{corollary effective Dyck}}.\smallskip

The equivalence between $(i)$ and $(iii)$ is clear from Theorem \ref{theorem Dyck}. The implication $(ii) \Longrightarrow (i)$ is immediate from Corollary \ref{corollary union spc stops at the church-kleene ordinal} and Lemma \ref{lemma basic properties of effective semipositive B}. 

For the implication $(i) \Longrightarrow (ii)$ we consider the function $v$ as in the statement of Theorem \ref{theorem uniformity version Dyck}. Since $A$ is $\del$ there are recursive points $\alpha$ and $\beta$, which are analytic codes for the sets $A$ and $\cantor \setminus A$ respectively. The only set, which separates $A$ from its complement is $A$ itself; therefore $v(\alpha,\beta)$ is a semi-positive code for $A$. Moreover, since $v$ is recursive we have that $v(\alpha,\beta)$ is recursive. It follows that $A$ is $\HYP$-semi-positive.

\section{The uniform Preiss Separation}

As noted in \cite{kechris_classical_dst} the convexly generated sets are exactly the ones, which are obtained from the open convex subsets of $\R^N$ by applying the operations of countable increasing union and countable intersection. In other words we can replace the term ``compact convex" with ``open convex". It is useful in the sequel to give the argument; first we notice that $K = \cap_{n \in \om} \set{x}{d(x,K) < 2^{-n}}$ for every closed $K \subseteq \R^N$, where $d(x,K) = \inf\set{\normR{x-y}}{y \in K}$, and $\normR{\cdot}$ is a fixed norm on $\R^N$. If moreover $K$ is convex then every set $\set{x}{d(x,K) < 2^{-n}}$ is convex (and open). Hence every closed convex set is the countable intersection of convex open sets. For the converse direction we assume that $U$ is an open convex subset of $\R^N$, and we consider a sequence $(B_i)_{i \in \om}$ of bounded open balls with $U = \cup_i B_i = \cup_i \overline{B_i}$. Put $K_i = \cup _{j \leq i} \; \overline{B_j}$. Then the \emph{convex hull} $H(K_i)$ is compact (this follows from a well-known result of Caratheodory, see for example \cite[28.13]{kechris_classical_dst}) and is also a convex set, which is contained in $U$, since the latter is a convex set which contains $K_i$. Hence $U$ is the increasing union of convex  compact sets.

\subsection{The constructive proof.}

Here we give our constructive proof to the Preiss Separation Theorem. As in the Dyck separation the idea is to define a well-founded tree, which reflects the fact that one of our given analytic sets is convex.  

We recall that a \emph{Souslin scheme} (also called \emph{system}) on \ca{X} is a family $(P_u)_{u \in \om^{<\om}}$ of subsets of \ca{X}. The \emph{Souslin operation} $\ca{A}$ applies to Souslin schemes $(P_u)_{u \in \om^{<\om}}$ and produces the sets
\[
\ca{A}_u P_u = \set{x \in \ca{X}}{(\exists \alpha)(\forall n)[x \in P_{\alpha \upharpoonright n}]}.
\]
It is a well-known fact in descriptive set theory that a set is analytic exactly when it has the form $\ca{A}_uP_u$ for some Souslin system $(P_u)_{u \in \om^<\om}$ of closed sets. Furthermore the scheme can be assumed to be \emph{regular}, \ie for all $u \sqsubseteq v$ we have $P_v \subseteq P_u$, and \emph{of vanishing diameter}, \ie for some compatible metric on the underlying space $\ca{X}$ and for all $\alpha \in \baire$ we have $\ds \lim_{n \to \infty}$diam$_d(P_{\alpha \upharpoonright n}) = 0$, where diam$_d$ is the usual diameter with respect to the metric $d$ - see \cite[25.7]{kechris_classical_dst}.

There are however cases, where we need some different properties for our Souslin scheme, and more specifically in our case we need to have that $P_{\cn{u}{(m)}} \subseteq P_{\cn{u}{(m+1)}}$. (This allows us to take increasing unions.) This is possible to satisfy with a regular Souslin system of \emph{analytic} sets, see \cite[25.13]{kechris_classical_dst}, but it is not implicit that these sets can be chosen to be closed. (In general $\ca{A}_u \overline{P_u}$ might be bigger than $\ca{A}_u P_u$.)

This poses an obstacle towards giving a constructive proof. To explain this better, we will need to consider the convex hull $H(K)$ of a given set $K \subseteq \R^N$, and make sure that $H(K)$ is a closed set. The standard argument for this, is to utilize the Caratheodory Theorem, which implies that $H(K)$ is compact, whenever $K$ is compact as well. In the proof by contradiction (see \cite{preiss_the_convex_generation_of_Borel_sets} and \cite[28.15]{kechris_classical_dst}) the set $K$ has the form $\cap_{n \in \om} P_{\alpha \upharpoonright n}$, and it is possible to choose the Souslin scheme so that all sets of this form are compact. The preceding $\alpha$ is obtained by repeating infinitely many times our hypothesis towards the contradiction. However in any proof, which uses well-founded trees, it is natural to expect that we will have only finitely many $P_u$'s in hand at any given stage. Since the latter sets are analytic in general, it is not obvious how to obtain the needed property of compactness.

The idea to overcome this difficulty is to bring each of the preceding analytic sets $P_u$ in the from $\ca{A}_vP^u_v$ for some suitably chosen double Souslin scheme $(P^u_v)_{u,v}$ of closed sets, and then use this to produce a regular Souslin scheme $(Q_u)_{u}$ of closed sets, which reflects (a somewhat stronger version of) the increasing property $P_{\cn{u}{(m)}} \subseteq P_{\cn{u}{(m+1)}}$ that we want to have. To obtain compactness we will simply restrict $Q_u$ on the cubes $[-m,m]^N$, where $m \in \om$. 

\begin{definition}\normalfont
\label{definition good Souslin code}
Let $\ca{X}$ be a Polish space and $A$ be an analytic subset of $\ca{X}$. We say that $(Q_u)_{u \in \om^{< \om}}$ is a \emph{good Souslin scheme for $A$} if we have the following properties:
\renewcommand{\theindex}{\alph{index}}
\begin{list}{\tu{(}\alph{index}\tu{)}}{\usecounter{index}}
\item\label{souslin good set equals to A} $A = \ca{A}_u Q_u$;
\item \label{souslin good closed} Each $Q_u$ is closed;
\item \label{souslin good regular}The system $(Q_u)_{u}$ is regular, \ie we have $Q_v \subseteq Q_u$ for all $u \sqsubseteq v$;
\item \label{souslin good system increasing}For all finite sequences $u,v$ and all $i \leq j$ it holds $Q_{\cn{u}{\cn{(i)}{v}}} \subseteq Q_{\cn{u}{\cn{(j)}{v}}}$.
\end{list}
\renewcommand{\theindex}{\arabic{index}}

A given $\beta \in \baire$ is a \emph{good Souslin code} if the Souslin scheme $(Q_u)_{u}$ defined by
\[
Q_{\dec{s}} = \fcode^\ca{X}((\beta)_{s})
\]
is good for $\ca{A}_uQ_u$, where $s \in \Seq \mapsto \dec{s}$ is the decoding function and $\fcode^\ca{X}$ is the fixed universal system for $\boldp^0_1 \upharpoonright \ca{X}$ from the Introduction.\smallskip
\end{definition}

\begin{proposition}
\label{proposition every analytic set admits a good Souslin scheme}
Every analytic set in a Polish space admits a good Souslin scheme.
\end{proposition}

\begin{proof}
Let \ca{X} be a Polish space and $A$ be an analytic subset of \ca{X}. The argument is an elaboration of the ones in \cite[25.7 and 25.13]{kechris_classical_dst} together with some minor variations in order to establish the effective arguments that we will need in the sequel.

Let $p$ be a compatible metric on $\ca{X}$. From \cite[1A.1]{yiannis_dst} there is a function $h: \om^{< \om} \to \ca{X}$ such that $p(h(u),h(\cn{u}{(i)})) < 2^{-\lh(u)}$ for all $u$, $i$ and the function $\pi: \baire \to \ca{X}$ defined by $\pi(\alpha) = \lim_{n \to \infty} h(\alpha \upharpoonright n)$ is a continuous surjection. In fact if $\ca{X}$ is recursive Polish, then $h$ and consequently $\pi$ can be chosen to be recursive.

We fix a recursive bijection $\pair{\cdot} : \om^2 \to \om$ with recursive inverse. Given a pair $(u,v)$ of a finite sequences of the same length $n$ we put (by a slight abuse of notation)
\[
\pair{u,v} = (\pair{u(0),v(0)},\dots,\pair{u(n-1),v(n-1)}).
\]
Clearly the latter is a finite sequence of length $n$.

Similarly given $(\alpha,\gamma) \in \baire \times \baire$ we put
\[
\pair{\alpha,\gamma} = (\pair{\alpha(0),\gamma(0)},\dots, \pair{\alpha(n),\gamma(n)},\dots).
\]
Since the function $\pair{\cdot}: \om^2 \to \om$ is bijective it is clear that for every $w \in \om^{<\om}$ there is a unique pair $(u,v) \in \om^{< \om}$ with $\lh(u) = \lh(v) = \lh(w)$ and $w = \pair{u,v}$. We denote the latter unique pair by $(w^1,w^2)$.

Similarly for all $\ep \in \baire$ there is a unique pair $(\ep^1,\ep^2) \in \baire \times \baire$ such that $\ep = \pair{\ep^1,\ep^2}$.

The set $B : = \pi^{-1}[A]$ is a $\sig$ subset of $\baire$, and so there is a recursive tree of pairs $T$ such that
\[
\pi(\alpha) \in A \iff \alpha \in B \iff (\exists \gamma)(\forall t)[(\alpha \upharpoonright t, \gamma \upharpoonright t) \in T].
\]
for all $\alpha \in \baire$.

We define the binary relation $\preceq$ on $\om^{< \om}$ by
\[
v \preceq u \ \iff \ \text{for all} \ i < \min \{\lh(u),\lh(v)\} \ \text{we have} \ v(i) \leq u(i),
\]
where $u, v \in \om^{<\om}$. Notice that for all $u$ there are infinitely many $v$'s for which $v \preceq u$, but there are only finitely many whose length is at most the length of $u$.

Given finite sequences $u$, $c$ not necessarily of the same length we define
\begin{align*}
x \in P^u_c \iff c \preceq u \quad \& \quad (c^1,c^2) \in T \quad \& \quad p(x,h(c^1)) \leq 2^{-\lh(c)+1}
\end{align*}
and
\[
Q_u = \bigcup \set{P^u_c}{c \preceq u\;, \ \lh(c) = \lh(u)} = \bigcup \set{P^u_c}{\lh(c) = \lh(u)}.\footnote{It can be proved that $A = \ca{A}_u \left(\ca{A}_c P^u_c\right)$ but we will not need to use this property directly.}
\]
We claim that $(Q_u)_{u \in \om^{< \om}}$ is a good Souslin scheme for $A$. It is clear that each $Q_u$ is closed as the finite union of closed sets.\smallskip

\emph{Claim 1.} For all $u,c,d \in \om^{< \om}$ with $c \sqsubseteq d$ we have $P^u_d \subseteq P^u_c$.\smallskip

\emph{Proof of the claim.} Without loss of generality we can assume that $d$ is an one-point extension of $c$, say $d = \cn{c}{(k)} = \cn{\pair{c^1,c^2}}{(\pair{k^1,k^2})}$. Clearly $d^1 = \cn{c^1}{(k^1)}$. Put $n = \lh(c)$ and suppose that $x$ is a member of $P^u_d$, \ie $d \preceq u$, $(d^1,d^2) \in T$ and $p(x,h(d^1)) < 2^{-\lh(d)+1} = 2^{-n}\;$. For all $i < \min\{\lh(c),\lh(u)\} \leq \min\{\lh(d),\lh(u)\}$ we have $c(i) = d(i) \leq u(i)$. Hence $c \preceq u$. Clearly $c^i \sqsubseteq d^i$, $i=1,2$, so $(c^1,c^2) \in T$. From the properties of $h$ we have
\begin{align*}
p(h(c^1),x)
\leq& \ p(h(c^1),h(\cn{c^1}{(k^1)})) + p(h(\cn{c^1}{(k^1)}),x)\\
<& \ 2^{-n} + p(h(d^1),x)\\
<& \ 2^{-n} + 2^{-n} = 2^{-n+1} = 2^{-\lh(c)+1}.
\end{align*}
Therefore $x \in P^u_c$.\smallskip

\emph{Claim 2.} For all $u,v,c \in \om^{< \om}$ with $u \sqsubseteq v$ we have $P^v_c \subseteq P^u_c$.\smallskip

\emph{Proof of the claim.} If $P^v_c$ is non-empty then $c \preceq v$ and so for all $i < \min\{\lh(c),\lh(u)\} \leq \min\{\lh(c),\lh(v)\}$ we have $c(i) \leq v(i) = u(i)$. It follows that $c \preceq u$ and from this it is immediate that $P^v_c \subseteq P^u_c$.\smallskip

By combining the preceding two claims we have for all $c \sqsubseteq d$ and all $u \sqsubseteq v$ that
\begin{align}\label{equation A proposition every analytic set admits a good Souslin scheme}
P^{v}_d \subseteq P^v_c \subseteq  P^u_c.
\end{align}

Now we can show the regularity of the Souslin scheme $(Q_u)_u$. Suppose that $u \sqsubseteq v$ and that $x \in Q_v$, \ie $x \in P^v_d$ for some $d \in \om^{< \om}$ with $d \preceq v$ and $\lh(d) = \lh(v)$. We then take the restriction $c: = d \upharpoonright \lh(u)$, so that $\lh(c) = \lh(u)$, $c \preceq u$, and hence $P^u_c \subseteq Q_u$. Using that $c \sqsubseteq d$ it follows from (\ref{equation A proposition every analytic set admits a good Souslin scheme}),
\[
x \in P^v_d \subseteq P^v_{c} \subseteq P^u_{c} \subseteq Q_u,
\]
\ie $Q_v \subseteq Q_u$.\smallskip

\emph{Claim 3.} It holds $P^{u}_c \subseteq P^{v}_c$ for all $c,u,v$ with $u \preceq v$ and $\lh(u) = \lh(v)$.\smallskip

\emph{Proof of the claim.} If $u \preceq v$ and $\lh(u) = \lh(v)$ then for all $c \preceq u$ and all $i < \min \{\lh(v),\lh(c)\} = \min \{\lh(u),\lh(c)\}$ it holds $c(i) \leq u(i) \leq v(i)$, and so $c \preceq v$. Hence $P^{u}_c \subseteq P^{v}_c$.\smallskip

Now we can show the property (\ref{souslin good system increasing}). Let $i \leq j$ and $u, v \in \om^{<\om}$. Then $\cn{u}{\cn{(i)}{v}} \preceq \cn{u}{\cn{(j)}{v}}$. Since the latter two sequences have the same length it follows from Claim 3 that $P^{\cn{u}{\cn{(i)}{v}}}_c \subseteq P^{\cn{u}{\cn{(j)}{v}}}_c$ for all $c$. Therefore
\begin{align*}
Q_{\cn{u}{\cn{(i)}{v}}} 
=& \ \bigcup \set{P^{\cn{u}{\cn{(i)}{v}}}_c}{\lh(c) = \lh(u) + \lh(v) + 1}\\
\subseteq& \ \bigcup \set{P^{\cn{u}{\cn{(j)}{v}}}_c}{\lh(c) = \lh(u) + \lh(v) + 1}\\
=& \ Q_{\cn{u}{\cn{(j)}{v}}}.
\end{align*}

It remains to show that $A = \ca{A}_u Q_u$. Suppose that $x$ is a member of $A$; using the surjectiveness of $\pi$ there is some $\gamma = \pair{\gamma^1,\gamma^2}$ such that $(\gamma^1,\gamma^2) \in [T]$ and $x = \pi(\gamma^1)$. It follows from the properties of $h$ that $p(x,h(\gamma^1 \upharpoonright n)) \leq 2^{-n+1}$ for all \n. 

Fix some $\n$ and put $c_n = u_n= \gamma \upharpoonright n$. Since $c_n^i = \gamma^i \upharpoonright n$, $i=1,2$, it follows that $(c^1_n, c^2_n) \in T$. Moreover $c_n \preceq u_n$, $\lh(c_n) = \lh(u_n) = n$, and also $p(x,h(c^1_n)) = p(x,h(\gamma^1 \upharpoonright n)) \leq 2^{-n+1}$. This means that $x \in P^{u_n}_{c_n} \subseteq Q_{u_n} = Q_{\gamma \upharpoonright n}$. Thus $x \in \cap_n Q_{\gamma \upharpoonright n} \subseteq \ca{A}_u Q_u$.

This settles the left-to-right inclusion. For the converse inclusion suppose that $x$ is a member of $\ca{A}_u Q_u$. Consider some $\alpha$ such that $x \in \cap_n Q_{\alpha \upharpoonright n}$. Thus for all $n$ there is some $c_n \preceq \alpha \upharpoonright n$ such that $x \in P^{\alpha \upharpoonright n}_{c_n}$ and $\lh(c_n) = \lh(\alpha \upharpoonright n) = n$. In particular $x \in \cap_n P^{\alpha \upharpoonright n}_{c_n}$. 

Define
\[
S(\alpha,x) = \set{c \in \om^{<\om}}{x \in P^{\alpha \upharpoonright \lh(c)}_c}.
\]
Clearly $c_n \in S(\alpha,x)$ for all $\n$, and so $S(\alpha,x)$ is non-empty, in fact it is an infinite set since $\lh(c_n) = n$.

If $d \in S(\alpha,x)$ and $c \sqsubseteq d$ it follows from (\ref{equation A proposition every analytic set admits a good Souslin scheme}) that
\[
x \in P^{\alpha \upharpoonright \lh(d)}_d \subseteq P^{\alpha \upharpoonright \lh(d)}_c \subseteq P^{\alpha \upharpoonright \lh(c)}_c.
\]
Hence $c \in S(\alpha,x)$, and so the latter is a tree. Since for all $\cn{c}{(k)} \in S(\alpha,x)$ it holds $k \leq \alpha(\lh(c))$, we have that $S(\alpha,x)$ is of finite branching. So by an application of K\"{o}nig's Lemma there is some $\gamma \in [S(\alpha,x)]$. In other words $x \in P^{\alpha \upharpoonright n}_{\gamma \upharpoonright n}$ for all $n$. In particular $(\gamma^1 \upharpoonright n, \gamma^2 \upharpoonright n) \in T$ for all $n$, and so $\pi(\gamma^1) \in A$. Moreover $p(x,h(\gamma^1 \upharpoonright n)) \leq 2^{-n+1}$ and so $x = \lim_{n \to \infty} h(\gamma^1 \upharpoonright n) = \pi(\gamma^1)$. Therefore $x = \pi(\gamma^1) \in A$.

So we showed that $A = \ca{A}_u Q_u$ and we are done.
\end{proof}\smallskip

Now we are ready to give our constructive \emph{proof to Theorem \ref{theorem preiss separation}.}\smallskip

We fix the distance function $p_\infty = ((x,y) \mapsto \normR{x-y}_{\infty})$ on $\R^N$, a recursive surjection $\pi: \baire \surj \R^N$ and a recursive function $h: \om^{< \om} \to \R^N$ such that $\normR{h(\alpha \upharpoonright (n+1)) - \pi(\alpha)}_{\infty} \leq 2^{-n+1}$, see \cite[1A.1]{yiannis_dst}.

Since $\pi^{-1}[B]$ is a $\sig$ subset of $\baire$ there is a recursive tree of pairs $S$ such that for all $\beta$,
\[
\beta \in \pi^{-1}[B] \iff (\exists \delta)(\forall t)[(\beta \upharpoonright t, \delta \upharpoonright n) \in S].
\]
Notice that from the surjectiveness of $\pi$ we have $(\pi \circ \pr)[[S]] = B$.

Suppose that $(Q_u)_{u \in \om^{<\om}}$ is a good Souslin scheme for $A$ (Proposition \ref{proposition every analytic set admits a good Souslin scheme}).

We take the sets
\begin{align*}
A_{u} = \ca{A}_v Q_{\cn{u}{v}}, \quad B_{(b,d)}= (\pi \circ \pr)[[S_{(b,d)}]],
\end{align*}
where $u, b, d \in \om^{< \om}$ with $\lh(b) = \lh(d)$. Notice that $B_{\emptyset} = (\pi \circ \pr)[[S]] = B$.\smallskip

\emph{Claim 1}. For all $u$, $(b,d)$, and $i \leq j$ it holds
\begin{align*}
B_{(b,d)} =& \ \cup_{(k,l) \in \om^2} B_{\cn{(b,d)}{(k,l)}}\;,\\
A_{u} =& \ \cup_{m \in \om} A_{\cn{u}{(m)}}\;,\\
A_{\cn{u}{(i)}} \subseteq& \ A_{\cn{u}{(j)}}\;.
\end{align*}

\emph{Proof of the claim.} 
The first two assertions are proved by direct computations. For all $x \in \R^N$ we have
\begin{align*}
x \in B_{(b,d)} 
\iff& \ (\exists \beta, \delta)[(\beta,\delta) \in [S_{(b,d)}] \ \& \ \pi(\beta) = x]\\
\iff& \ (\exists k,l)(\exists \beta,\delta)[(\beta,\delta) \in [S_{\cn{(b,d)}{(k,l)}}] \ \& \ \pi(\beta) = x]\\
\iff& \ (\exists k,l)[x \in B_{\cn{(b,d)}{(k,l)}}];
\end{align*}
and
\begin{align*}
x \in A_u 
\iff& \ (\exists \gamma)(\forall n)[x \in Q_{\cn{u}{\gamma \upharpoonright n}}]\\
\iff& \ (\exists m)(\exists \gamma^\ast)(\forall n)[x \in Q_{\cn{u}{\cn{(m)}{\gamma^\ast \upharpoonright n}}}]\\
\iff& \ (\exists m)[x \in \ca{A}_v Q_{\cn{u}{\cn{(m)}{v}}}]\\
\iff& \ (\exists m)[x \in A_{\cn{u}{(m)}}].
\end{align*}
For the inclusion we have from the property (\ref{souslin good system increasing}) of the definition of a good Souslin scheme that $Q_{\cn{u}{\cn{(i)}{v}}} 
\subseteq Q_{\cn{u}{\cn{(j)}{v}}}\;$ for all $v \in \om^{< \om}$,
from which it follows
\[
A_{\cn{u}{(i)}} = \ca{A}_v Q_{\cn{u}{\cn{(i)}{v}}} \subseteq \ca{A}_v Q_{\cn{u}{\cn{(j)}{v}}} = A_{\cn{u}{(j)}}.
\]
This finishes the proof of the claim.\smallskip

\emph{Definition of $J$.} Next we define the tree $J$ as the direct sum of trees of tuples as follows: given $m \in \om$ and $b,d,u \in \om^{< \om} $ of the same length $n \in \om$ put
\begin{align*}
(m,b,d,u) \in J
\iff& \ \hspace*{5mm} Q_{u} \cap [-m,m]^N \neq \emptyset\\
& \ \hspace*{2mm}\& \ (b,d) \in S\\
& \ \hspace*{2mm}\& \ (\forall s \leq n)[p_\infty\left(h(b \upharpoonright s),H\left(Q_{u \upharpoonright s} \cap [-m,m]^N\right)\right) < 2^{-s+3}]\;, 
\end{align*}
where $H(P)$ is the convex hull of $P$. In the preceding definition we allow $b = d = u= \emptyset$, so that $(m) \in J$ for all $m \in \om$; we also include the empty sequence in $J$ as well. Using that $Q_v \subseteq Q_u$ for $u \sqsubseteq v$ it is immediate that $J$ is indeed a tree.\smallskip

\smallskip

\emph{Claim 2.} The tree $J$ is well-founded.\smallskip

\emph{Proof of the claim.} Suppose towards a contradiction that there is an infinite branch $(m,\beta,\delta,\ep) \in [J]$. Then in particular $(\beta,\gamma) \in [S]$, so that $\pi(\beta) \in B$.

For all $n$ there is some $y_n \in H(Q_{\ep \upharpoonright n} \cap [-m,m]^N)$ such that $\normR{h(\beta \upharpoonright n) - y_n}_{\infty} < 2^{-n+3}$. It follows that $\pi(\beta) = \lim_n h(\beta \upharpoonright n) = \lim_n y_n$.

For all $n$, from the Caratheodory Theorem there are $\lambda^n_k \in [0,1]$ and $z^n_k \in Q_{\ep \upharpoonright n} \cap [-m,m]^N$, $k=1,\dots,N+1$, with $\sum_{k=1}^{N+1}\lambda^n_k =1$ and $y_n = \sum_{k=1}^{N+1}\lambda^n_k \cdot z^n_k$. Since the sets $[-m,m]^N$ and $[0,1]$ are compact, there is a strictly increasing sequence of naturals $(i_n)_{\n}$ and there are points $z_1,\dots,z_{N+1} \in [-m,m]^N$, $\lambda_1, \dots, \lambda_{N+1} \in [0,1]$ such that
\begin{align*}
z^{i_n}_k \to z_k, \quad \lambda^{i_n}_k \to \lambda_k,
\end{align*} 
for all $k=1,\dots,N+1$. Hence $y_{i_n} =  \sum_{k=1}^{N+1} \lambda^{i_n}_k \cdot z^{i_n}_k \to \sum_{k=1}^{N+1} \lambda_k \cdot z_k$, and so $\pi(\beta) =  \lim_{n} y_{i_n} = \sum_{k=1}^{N+1} \lambda_k \cdot z_k$. It is also clear that $\sum_{k=1}^{N+1} \lambda_k = \lim_{n \to \infty} \sum_{k=1}^{N+1} \lambda^{i_n}_k = 1$.

Given $k \in \{1,2\dots,N+1\}$ using the regularity of $(Q_u)_u$ we have that $z^{i_n}_k \in Q_{\ep \upharpoonright i_n} \subseteq Q_{\ep \upharpoonright t}$ for all $n \geq t$. Since the $Q_u$'s are closed sets it follows that $z_k = \lim_{n \to \infty} z^{i_n}_k \in Q_{\ep \upharpoonright t}$ for all $t$. Hence $z_k \in \cap_t Q_{\ep \upharpoonright t} \subseteq \ca{A}_u Q_u = A$. In other words all $z_k$'s are members of $A$.

Thus $\pi(\beta) =  \sum_{k=1}^{N+1} \lambda_k \cdot z_k$ is the convex combination of elements of $A$, and since the latter  is a convex set, it follows that $\pi(\beta) \in A$, contradicting that the sets $A$ and $B$ are disjoint. This completes the proof that $J$ is well-founded.\bigskip

In the sequel we need to consider the restrictions on the cubes $[-m,m]^N$,
\[
A_u^m := A_u = \cap [-m,m]^N\;, \quad Q_u^m: = Q_u \cap [-m,m]^N. 
\]

We will define a family $(C_\sigma)_{\sigma \in J \setminus \{\emptyset\}}$ of subsets of $\R^N$ such that for all non-empty $\sigma \in J$ of the form $(m,b,d,u)$ (with possibly $b =d = u = \emptyset$) it holds:\smallskip

(a) $C_\sigma$ is convexly generated;\smallskip

(b) $C_\sigma$ separates $A^m_u$ from $B_{(b,d)}$.\bigskip

If we do this we then define
\[
C_\emptyset = \cup_{m \in \om} \cap_{s \geq m} C_{(s)}.
\]
Using (a) we can see that $C_\emptyset$ is the countable increasing union of convexly generated sets, and so it is convexly generated as well. Moreover by applying (b) to $\sigma = (m)$ we have that $C_{(m)}$ separates $A^m_{\emptyset}$ from $B_\emptyset = B$. Since $A^m_\emptyset \subseteq A^{m+1}_\emptyset$ it follows that the set $\cap_{s \geq m} C_{(s)}$ also separates $A^m_\emptyset$ from $B$ for all $m \in \om$, and also that the union $\cup_{m \in \om} \cap_{s \geq m} C_{(s)} = C_\emptyset$ separates $\cup_{m \in \om} A^m_\emptyset = A_\emptyset = A$ from $B$. Thus it suffices to define $(C_\sigma)_{\sigma \in J \setminus \{\emptyset\}}$ as above.\bigskip

\emph{Claim 3.} Suppose that $\sigma =(m,b,d,u)$ is a member of $J$ and that $(D^\sigma_{(k,l,i)})_{k,l,i \in \om}\;$ is a family of sets in $\R^N$ such that for all $k,l,i \in \om$:
\begin{list}{(\alph{index}*)}{\usecounter{index}}
\item $D^\sigma_{(k,l,i)}$ is convexly generated;

\item $D^\sigma_{(k,l,i)}$ separates $A^m_{\cn{u}{(i)}}$ from $B_{\cn{(b,d)}{(k,l)}}$.
\end{list}

Put
\begin{align*}
C^j_\sigma :=& \ \cap_{i \geq j} \cap_{(k,l) \in \om^2} D^\sigma_{(k,l,i)}\\
C_\sigma : =& \ \cup_{j \in \om} C^j_\sigma.
\end{align*}

Then the set $C_\sigma$ satisfies the preceding properties (a) and (b).\smallskip

\emph{Proof of the claim.}
Notice that for all $j$ the set $C^j_\sigma$ is convexly generated and $C^j_\sigma \subseteq C^{j+1}_\sigma$. So the union $\cup_{j \in \om} C^j_\sigma = C_\sigma$ is convexly generated as the countable increasing union of convexly generated sets.

Now we show that $C_\sigma$ separates $A^m_{u}$ from $B_{(b,d)}$. Suppose that $x \in A^m_u$. From Claim 1 there  is some $j \in \om$ such that $x \in A^m_{\cn{u}{(j)}}$, and moreover $x \in A^m_{\cn{u}{(i)}}$ for all $i \geq j$. Using (b*) above it follows that $x \in D^{\sigma}_{(k,l,i)}$ for all $i \geq j$ and all $k,l$. In other words $x \in C^j_\sigma \subseteq C_\sigma$. If there were some $x \in C_\sigma \cap B_{(b,d)}$, then by using again Claim 1 there would be naturals $j,k,l$ such that $x \in C^j_\sigma \cap B_{\cn{(b,d)}{(k,l)}} \subseteq D^\sigma_{(k,l,j)} \cap B_{\cn{(b,d)}{(k,l)}}$, contradicting (b*). Hence $C_\sigma \cap B_{(b,d)} = \emptyset$.\smallskip

It remains to define a family $(D^\sigma_{(k,l,j)})_{k,l,j \in \om, \sigma \in J \setminus \{\emptyset\}}\;$, which satisfies the properties (a*) and (b*) above. 

We do this by bar recursion on $J$. Let $\sigma = (m,b,d,u) \in J$ and let $k,l,j,n \in \om$ with $n =\lh(b) = \lh(d) = \lh(u)$.\smallskip

\emph{Case 1.} We have $\cn{\sigma}{(k,l,j)} \not \in J$. We consider the following sub-cases.\smallskip

\emph{Case 1a.} It holds $Q^m_{\cn{u}{(j)}} = \emptyset$. In this case, since $A_w \subseteq Q_w$ for all $w$ (from the regularity of $(Q_v)_v$), we have in particular that $A^m_{\cn{u}{(j)}} = \emptyset$. So we take $D^\sigma_{(k,l,j)} = \emptyset$, which is clearly convexly generated and separates $A^m_{\cn{u}{(j)}}$ from $B_{\cn{(b,d)}{(k,l)}}$.\smallskip

\emph{Case 1b.} It holds $Q^m_{\cn{u}{(j)}} \neq \emptyset$ and $\cn{(b,d)}{(k,l)} \not \in S$. Then $B_{\cn{(b,d)}{(k,l)}} = (\pi \circ \pr)[[S_{\cn{(b,d)}{(k,l)}}]] = \emptyset$. Therefore we take $D^\sigma_{(k,l,j)} = \R^N$. This is a convexly generated set and separates $A^m_{\cn{u}{(j)}}$ from $B_{\cn{(b,d)}{(k,l)}}$.\smallskip

\emph{Case 1c.} It holds
\[
Q^m_{\cn{u}{(j)}} \neq \emptyset \quad \text{and} \quad \cn{(b,d)}{(k,l)} \in S\;.
\]
From this and our hypothesis that $\sigma = (m,b,d,u) \in J$ it follows that
\[
p_\infty\left(h(\cn{b}{(k)}),H\left(Q^m_{\cn{u}{(j)}}\right)\right) \geq 2^{-n+2} = 2^{-\lh(\cn{u}{(j)}) + 3}\;.
\]
We then take
\begin{align}\label{equation definition of Dsigma}
D^\sigma_{(k,l,j)} = \set{x \in \R^N}{p_\infty\left(x,H(Q^m_{\cn{u}{(j)}})\right) < 2^{-n}}
\end{align}
It is clear that $D^\sigma_{(k,l,j)}$ is an open subset of $\R^N$ and it is also easy to see that $D^\sigma_{(k,l,j)}$ is convex. It is thus convexly generated.\smallskip

We show that the latter set separates $A^m_{\cn{u}{(j)}}$ from $B_{\cn{(b,d)}{(k,l)}}$. Since $A_{\cn{u}{(j)}} \subseteq Q_{\cn{u}{(j)}}$ it follows that $p_\infty\left(x,H(Q^m_{\cn{u}{(j)}})\right) = 0 < 2^{-n}$ for all $x \in A^m_{\cn{u}{(j)}}$. Hence $A_{\cn{u}{(j)}} \subseteq D^\sigma_{(k,l,j)}$. Now assume that $x$ is a member of $B_{\cn{(b,d)}{(k,l)}}$. Then there is some $\beta \in \baire$ with $\cn{b}{(k)}\sqsubseteq \beta$ and $x = \pi(\beta)$. We observe that 
\[
p_{\infty}(\pi(\beta),h(\cn{b}{(k)})) = p_{\infty}(\pi(\beta),h(\beta \upharpoonright (n+1)))\leq 2^{-n+1},
\]
and so
\[
p_\infty\left(x,H(Q^m_{\cn{u}{(j)}})\right) \geq p_\infty\left(h(\cn{b}{(k)}),H(Q^m_{\cn{u}{(j)})}\right) - p_\infty(h(\cn{b}{(k)}),x) \geq 2^{-n+2} - 2^{-n+1} > 2^{-n}.
\]
Therefore $x \not \in D^\sigma_{(k,l,j)}$, \ie $D^\sigma_{(k,l,j)} \cap B_{\cn{(b,d)}{(k,l)}} = \emptyset$.\smallskip

\emph{Case 2.} We have $\cn{\sigma}{(k,l,j)} \in J$. From our inductive hypothesis the family $(D^{\cn{\sigma}{(k,l,j)}}_{(k_1,l_1,j_1)})_{k_1,l_1,j_1 \in \om}$ has been defined and satisfies the properties (a*) and (b*) above. Then we take
\begin{align*}
D^{\sigma}_{(k,l,j)} = \cup_{j_1} \cap_{i_1 \geq j_1} \cap_{k_1,l_1} \; D^{\cn{\sigma}{(k,l,j)}}_{(j_1,k_1,l_1)} \;.
\end{align*}
In other words $D^{\sigma}_{(k,l,j)} = C_{\cn{\sigma}{(k,l,j)}}$ in the notation of Claim 3. From the latter and our inductive hypothesis the set $C_{\cn{\sigma}{(k,l,j)}}$ is convexly generated and separates $A^m_{\cn{u}{(j)}}$ from $B_{\cn{(b,d)}{(k,l)}}$. This completes the construction.

\subsection{From analytic codes to Souslin codes uniformly.}

The preceding proof shows how to obtain a uniformity function in the Preiss Separation Theorem under the following assumptions.
\begin{list}{$\bullet$}{} 
\item The convex set $A$ is encoded in terms of a good Souslin scheme $(Q_u)_{u}$ for $A$.

\item We have a code for the clauses of the conjunctions defining the tree $J$ - this settles the complexity of $J$ and of the case distinction.

\item We have a way for obtaining codes of sets of the form $\set{x \in \R^N}{d(x, H(Q_u \cap [-m,m]^N)) <  2^{-n}}$ out of code for $(Q_u)_{u}$.
\end{list}

We will see that all of preceding assumptions can be derived in a $\del$-way from any analytic code for the convex set $A$. Here we deal with the first one - the other two assumptions are settled in \ref{subsection codes for some open convex sets} and \ref{subsection convexly generated codes and the proof to the uniform Preiss Separation Theorem}.

\begin{proposition}
\label{proposition from analytic code to good Souslin code}
Suppose that $\ca{X}$ is a recursive Polish space. Then every $\sig(\alpha)$ set $A \subseteq \ca{X}$ has a good Souslin code $\beta \in \baire$, which is recursive in $\alpha$.\smallskip

In fact the preceding assertion holds $\Sigma^0_1$-uniformly, \ie there is a recursive function $\strans^\ca{X}: \baire \to \baire$ such that for all $\alpha \in \baire$ the $\alpha$-recursive point $\strans^\ca{X}(\alpha)$ is a good Souslin code for the $\sig(\alpha)$ set $\scode^\ca{X}_1(\alpha)$.
\end{proposition}

\begin{proof}
We go back to the proof of Proposition \ref{proposition from analytic code to good Souslin code}, and we show that the $Q_u$'s can be obtained uniformly in the code of $A$.

As before let $p$ be a compatible metric on $\ca{X}$, and $h: \om^{< \om} \to \ca{X}$ be a recursive function such that $p(h(u),h(\cn{u}{(i)})) \leq 2^{-\lh(u)}$ for all $u$, $i$, and the function $f: \baire \to \ca{X}$ defined by $f(\alpha) = \lim_{n \to \infty} h(\alpha \upharpoonright n)$ is a recursive surjection.

We also consider the recursive isomorphism $\pair{\cdot} : \om^2 \to \om$ and the relation $\preceq$, which is  clearly recursive as well.

We take the universal set $\scode_1^{\ca{X}} \subseteq \baire \times \ca{X}$ for $\bolds^1_1 \upharpoonright \ca{X}$ and we define
\[
\tilde{B}(\beta,\alpha) \iff \scode_1^{\ca{X}}(\beta,f(\alpha)).
\]
Clearly $\tilde{B}$ is a $\sig$ subset of $\baire \times \baire$, so there is a tree of triples $\tilde{T}$ such that
\[
\tilde{B}(\beta,\alpha) \iff (\exists \gamma)(\forall t)[(\beta \upharpoonright t, \alpha \upharpoonright t, \gamma \upharpoonright t) \in \tilde{T}].
\]
We denote by $\tilde{T}(\beta)$ the tree of pairs $\set{(u,v)}{(\beta \upharpoonright \lh(u), u,v) \in \tilde{T}}$, so that 
\[
f(\alpha) \in \scode_1^{\ca{X}}(\beta)\iff \alpha \in \tilde{B}(\alpha) \iff (\exists \gamma)(\forall t)[(\alpha \upharpoonright t, \gamma \upharpoonright t) \in \tilde{T}(\beta)].
\]
In other words we replace $A$ and $T$ from above with $\scode_1^{\ca{X}}(\beta)$ and $\tilde{T}(\beta)$ respectively.

We define
\begin{align*}
\tilde{P}(\beta,t,s,x) 
\iff& \dec{s} \preceq \dec{t} \ \& \ (\dec{s}^1,\dec{s}^2) \in \tilde{T}(\beta) \ \& \ p(x,h(\dec{s}^1)) \leq 2^{-\lh(\dec{s})+1}\;;\\
R(s,t) \iff& \ \dec{t} \preceq \dec{s} \ \& \ \lh(\dec{t}) = \lh(\dec{s}).
\end{align*}
Clearly $R$ is recursive and moreover there is a recursive function $u(s)$ such that for all $t,s$ if $R(s,t)$ holds then $t \leq u(s)$. 

Now we take the set $\tilde{Q}$ defined by
\[
\tilde{Q}(\beta, s, x) \iff (\exists t \leq u(s))[R(s,t) \ \& \ \tilde{P}(\beta,t,s,x)].
\]
We denote the $(\beta,s)$-section of $\tilde{Q}$ by $Q_{\dec{s}}(\beta)$. It is now clear from the proof of Proposition \ref{proposition from analytic code to good Souslin code} that the Souslin  scheme $(Q_{\dec{s}}(\beta))_{s \in \Seq}$ is a good Souslin scheme for the analytic set $\scode^{\ca{X}}_1(\beta)$.

It remains to extract a code for $(Q_{\dec{s}}(\beta))_{s}$ uniformly on $\beta$. Clearly $\tilde{Q}$ is $\Pi^0_1$, so there is some recursive $\ep \in \baire$ such that $\tilde{Q} = \fcode^{\baire \times \om \times \ca{X}}(\ep)$. Then for some recursive function $S$ we have that
\[
(\beta,s,x) \in \tilde{Q} \iff \fcode^{\om \times \ca{X}}(\ep,\beta,s,x) \iff \fcode^\ca{X}(S(\ep,\beta,s),x)
\]
for all $\beta,s,x$. Therefore $S(\ep,\beta,s)$ is a closed code for the $(\beta,s)$-section of $\tilde{Q}$, which is exactly $Q_{\dec{s}}(\beta)$.

Finally we define $\strans^\ca{X}: \baire \to \baire$ by
\[
\strans^\ca{X}(\beta)(\langle s, j \rangle) = S(\ep,\beta,s)(j)
\]
and $\strans^\ca{X}(k) = 0$ if $k$ does not have the form $\langle s,j \rangle$, so that $(\strans^\ca{X}(\beta))_{s} = S(\ep,\beta,s)$ for all $s$, \ie $(\strans^\ca{X}(\beta))_{s}$ is a closed code for $\tilde{Q}_{\dec{s}}(\beta)$. Clearly $\strans^{\ca{X}}$ is recursive. 
\end{proof}

Although not necessary for the uniform Preiss Separation Theorem, it is natural to ask if the analogous result can be given for a somewhat more standard class of Souslin schemes.

\begin{definition}\normalfont
\label{definition regular closed-Souslin code}
Suppose that $\ca{X}$ is a Polish space and that $A$ is an analytic subset of \ca{X}. We say that a Souslin scheme $(P_u)_{u \in \om^{< \om}}$ of subsets of \ca{X} is \emph{normal for $A$} if $A = \ca{A}_u P_u$, each $P_u$ is closed, and $(P_u)$ is regular and of vanishing diameter.  

A \emph{normal Souslin code} is some $\alpha \in \baire$, for which the Souslin scheme $(P_{\dec{s}})_{s \in \Seq} = (\fcode^{\ca{X}}((\alpha)_s))_{s \in \Seq}$ is normal for $\ca{A}_u P_u$.
\end{definition}

As it well-known every analytic subset of a Polish space admits a normal Souslin scheme \cite[25.7]{kechris_classical_dst}. This is fairly easy to do: given an analytic set $A$ we consider a continuous surjection $f: \baire \surj A$ and we take $P_u$ to be the closure of $f[\set{\alpha \in \baire}{u \sqsubseteq \alpha}]$, where $u \in \om^{< \om}$. This argument does not effectivize for the following reasons: (a) the preceding function $f$ cannot be chosen in general to be recursive even if $A$ is a $\Pi^0_1$ set, and (b) the closure of a $\sig$ or even $\del$ set is not necessarily \del, and thus it is not clear how to obtain a closed code in a $\del$ way.  

To settle (a) we use trees of pairs as in the proof of Proposition \ref{proposition every analytic set admits a good Souslin scheme}. Regarding (b) we give a condition for the space \ca{X}, which is satisfied for $\ca{X} = \R^N$ and under which the transition in the codes can be done effectively modulo some \del \ parameter.

\begin{proposition}
\label{proposition transition to normal Souslin codes}
Suppose that \ca{X} is a recursive Polish space and that there is a recursive surjection $f: \baire \surj \ca{X}$ with the property that the relation $Q \subseteq \om \times \ca{X}$ defined by
\[
Q(t,x) \iff x \in \overline{f[N(\baire,t)]}
\]
is $\del$.\smallskip

Then there is some $\ep \in \del$ and a $\Sigma^0_1(\ep)$-recursive function $\strans^\ca{X}_\ast: \baire \to \baire$ such that for all $\alpha \in \baire$ the $\del(\alpha)$-point $\strans^\ca{X}_\ast(\alpha)$ is a normal Souslin code for the $\sig(\alpha)$ set $\scode^\ca{X}_1(\alpha)$.
\end{proposition}

\begin{proof}
We fix a $\sig$ set $A$ for the discussion. We denote by $N_u$ the usual neighborhood $\set{\alpha \in \baire}{u \sqsubseteq \alpha}$ of the Baire space, where  $u \in \om^{<\om}$. 
 
We take the set $B : = f^{-1}[A]$; clearly $B$ is $\sig$ and so there is a recursive tree of pairs $T$ such that
\[
f(\alpha) \in A \iff \alpha \in B \iff (\exists \gamma)(\forall t)[(\alpha \upharpoonright t, \gamma \upharpoonright t) \in T].
\]
for all $\alpha \in \baire$. 

We consider the pairing function $\pair{\cdot} : \om^2 \to \om$ from the proof of Proposition \ref{proposition every analytic set admits a good Souslin scheme} and we keep the same notation $\pair{u,v}$, $\pair{\alpha,\gamma}$ for $u,v \in \om^{\om}$ and $\alpha,\gamma \in \baire$.

We define
\begin{align}\label{equation definition of Souslin scheme A}
P_{\pair{u,v}} =
\begin{cases}
\overline{f[N_u]}, & \ \text{if} \ (u,v) \in T,\\
\emptyset, & \ \text{else}.  
\end{cases}
\end{align}
To show the regularity of $(P_w)_{w}$, suppose that $w, w'$ are such that $w \sqsubseteq w'$. If $w' = \pair{u', v'}$ with $(u',v') \not \in T$ then $P_{w'} = \emptyset \subseteq P_w$. Hence we assume that $w' = \pair{u', v'}$ and that $(u', v') \in T$. Then there are  $u \sqsubseteq u'$ and $v \sqsubseteq v'$ such that $w = \pair{u,v}$. We have in particular that $(u,v) \in T$ and so $P_{w'} = \overline{f[N_{u'}]} \subseteq \overline{f[N_u]} = P_w$.

The fact that $(P_w)_{w}$ is of vanishing diameter follows from the continuity of $f$. Using the latter property and the surjectiveness of $f$ it is not hard to prove that 
\begin{align*}
x \in A 
\iff& \ (\exists \alpha,\gamma)(\forall t)[x \in P_{\pair {\alpha \upharpoonright t, \gamma \upharpoonright t}}]\\
\iff& \ (\exists \delta)(\forall t)[x \in P_{\delta \upharpoonright t}].
\end{align*}
Hence $A = \ca{A}_w P_w$, and so $(P_w)_w$ is a normal Souslin scheme for $A$.

Now we proceed to the definition of $\strans^\ca{X}$. We consider the universal set $\scode_1^{\ca{X}} \subseteq \baire \times \ca{X}$ for $\bolds^1_1 \upharpoonright \ca{X}$ and as above we find a recursive tree of triples $\tilde{T}$ such that
\[
f(\alpha) \in \scode_1^{\ca{X}}(\beta)\iff (\exists \gamma)(\forall t)[(\alpha \upharpoonright t, \gamma \upharpoonright t) \in \tilde{T}(\beta)],
\]
where 
 $\tilde{T}(\beta)$ is the tree of pairs $\set{(u,v)}{(\beta \upharpoonright \lh(u), u,v) \in \tilde{T}}$. In other words we replace $A$ and $T$ from above with $\scode_1^{\ca{X}}(\beta)$ and $\tilde{T}(\beta)$ respectively.

Following (\ref{equation definition of Souslin scheme A}) we define the set
\begin{align*}
\tilde{P}(\beta,s,x) 
\iff& \ (\exists u,v)[ \ \lh(u) = \lh(v) \ \& \ \dec{s} = \langle u,v \rangle \ \& \ (u,v) \in \tilde{T}(\beta) \ \& \ x \in \overline{f[N_{u}]} \ ].
\end{align*}
We denote the $(\beta,s)$-section of $\tilde{P}$ by $P^\beta_{\dec{s}}$, and clearly we have
\[
P^\beta_{\dec{s}} =
\begin{cases}
\overline{f[N_u]}, & \ \text{if} \ \dec{s} = \langle u,v\rangle \ \& \ (u,v) \in \tilde{T}(\beta),\\
\emptyset, & \ \text{else}.
\end{cases}
\]
According to the preceding the Souslin scheme $(P^\beta_{\dec{s}})_s$ is a normal Souslin scheme for $\scode^\ca{X}_1(\beta)$ for all $\beta \in \baire$.

It remains to extract a code for $(P^\beta_{\dec{s}})_s$ uniformly on $\beta$. To do this we show first that $\tilde{P}$ is $\del$ and closed. The former assertion is immediate from our hypothesis about $f$: we note that every basic neighborhood $N(\baire,t)$ is of the form $N_{\dec{s_t}}$ where the map $t \mapsto s_t$ is recursive, and conversely, every $N_{\dec{s}}$ is of the form $N(\baire,t_s)$ for some recursive map $s \mapsto t_s$. So our hypothesis about the set $Q$ of the statement, translates to saying that the set
\[
Q'(s,x) \iff x \in \overline{f[N_{\dec{s}}]}
\]
is \del. It is then clear that
\begin{align*}
\tilde{P}(\beta,s,x) 
\iff& \ (\exists s_0,s_1)[ \ \lh(\dec{s_0}) = \lh(\dec{s_1}) \ \& \ \dec{s} = \langle \dec{s_0},\dec{s_1} \rangle \ \&\\
& \ \& \ (\beta \upharpoonright \lh(\dec{s_0}),\dec{s_0},\dec{s_1}) \in \tilde{T} \ \& \ Q'(\dec{s_0},x) \ ].
\end{align*}
From the latter equivalence it follows easily that $\tilde{P}$ is closed.

From \cite[4C.13]{yiannis_dst} (which is the basis of the Louveau Separation \cite{louveau_a_separation_theorem_for_sigma_sets}), there is some $\ep \in \del$ such that $\tilde{P} = \fcode^{\baire \times \om \times \ca{X}}(\ep)$. Then for some recursive function $S$ we have that
\[
(\beta,s,x) \in \tilde{P} \iff \fcode^{\om \times \ca{X}}(\ep,\beta,s,x) \iff \fcode^\ca{X}(S(\ep,\beta,s),x)
\]
for all $\beta,s,x$. Therefore $S(\ep,\beta,s)$ is a closed code for the $(\beta,s)$-section of $\tilde{P}$, which is exactly $\tilde{P}^\beta_\dec{s}$.

To conclude the proof we define $\strans^\ca{X}_\ast: \baire \to \baire$ by
\[
\strans^\ca{X}_\ast(\beta)(\langle s,j \rangle) = S(\ep,\beta,s)(j)
\]
and $\strans^\ca{X}_\ast(k) = 0$ if $k$ does not have the form $\langle s,j \rangle$, so that $(\strans_\ast^\ca{X}(\beta))_s = S(\ep,\beta,s)$ for all $s$, \ie $(\strans_\ast^\ca{X}(\beta))_s$ is a closed code for $\tilde{P}^\beta_{\dec{s}}$. Clearly $\strans_\ast^{\ca{X}}$ is $\ep$-recursive.  
\end{proof}

\begin{lemma}[Folklore]
\label{lemma surjection to Rn convex}
For every natural $N \geq 1$ there is a recursive bijection $\pi_N: \baire \to \R^N$ such that the set $\pi[N(\baire,t)]$ is a product of semi-open intervals,
\[
[p_1(t),q_1(t)) \times \dots \times [p_N(t),q_N(t))
\]
where $p_1(t),q_1(t)$,$\dots$,$p_N(t),q_N(t) \in \Q$. \smallskip 

Moreover the preceding rational numbers can be chosen so that the function $$t \mapsto (p_1(t),q_1(t),\dots, p_N(t),q_N(t)) \in \R^{2N}$$ is recursive.
\end{lemma}

\begin{proof}
First we notice that for real numbers $a < b$ it holds $[a,b) = \cup_{\n} \; [a_n,a_{n+1})$, where the sequence $(a_n)_{\n}$ is given by
\begin{align*}
a_0 =& \ a\;,\\
a_{n+1} =& \ a_n + (b-a) \cdot 2^{-(n+1)} = a + (b-a) \cdot \sum_{k=1}^{n+1} 2^{-k}\;.
\end{align*}
Clearly the intervals $[a_n,a_{n+1})$, \n, are pairwise disjoint, their length is at most $2^{-1}\cdot (b-a)$, and moreover $[a_n,a_{n+1}] \subseteq [a,b)$ for all $n$. It is also clear that if the numbers $a, b$ are rationals, then $(a_n)_{\n}$ is also a sequence of rational numbers. Additionally $(a_n)_{\n}$ is obtained recursively from the pair $(a,b)$.\smallskip

Then we find a Souslin scheme $(I_u)_{u \in \om^{<\om}}$ in $\R$ with the following properties:
\begin{list}{\tu{(\arabic{index})}}{\usecounter{index}}
\item $I_\emptyset = \R$;
\item $I_{u} = \cup_{\n} \; I_{\cn{u}{(n)}} = \cup_{\n} \ \overline{I_{\cn{u}{(n)}}}$;
\item length$(I_{\cn{u}{(n)}}) \leq 2^{-\lh(u)}$ for all $u,n$;
\item $I_u \cap I_v = \emptyset$ if $u$ and $v$ are incomparable;
\item $I_{u}$ is an interval of the form $[p,q)$, where $p, q \in \Q$ and $u \neq \emptyset$. Moreover the relations $P1, P_2 \subseteq \om \times \om$ defined by
\begin{align*}
P(t,s) \iff& \ \dec{t} \neq \emptyset \ \& \ \pq{s} \ \text{is the minimum of} \ I_{\dec{t}},\\ 
Q(t,s) \iff& \ \dec{t} \neq \emptyset \ \& \ \pq{s} \ \text{is the supremum of} \ I_{\dec{t}},
\end{align*}
are recursive.
\end{list}
This is fairly easy to do. In the first two steps we take $I_{\emptyset} = \R$, $I_{(2k)} = [k,k+1)$, and $I_{(2k+1)} = [-(k+1),-k)$, where $k \in \om$. In the next steps we employ the construction of the sequence $(a_n)_{\n}$ from above, for each of the intervals $I_u$ that have already been constructed. To ensure the latter of the preceding properties we apply Kleene's Recursion Theorem. We omit the details.\smallskip

Next we define the function $\sigma: \baire \to \R$ by
\[
\sigma(\alpha) = \ \text{the unique point in the intersection} \ \cap_{\n} \; I_{\alpha \upharpoonright n}.
\]
It is evident that $\sigma$ is bijective. We also have that
\[
\sigma(\alpha) \in (a,b) \iff (\exists u \neq \emptyset)[ u \sqsubseteq \alpha \ \& \ a < \min I_u \ \& \ \sup I_u < b ],
\]
which shows that $\sigma$ is recursive as well. Moreover it is not difficult to see that $\sigma[N_u] = I_u$ for all $u$.

We consider some natural number $N \geq 1$ and a recursive bijection $\bijN{\cdot}: \om^N \to \om: (a_1,\dots,a_N) \mapsto \bijN{a_1,\dots,a_N}$ whose inverse is also recursive. Similarly to above we denote the $N$-tuples $(u^1,\dots,u^N)$ of finite sequences of naturals of the same length $n$ with the finite sequence
\[
(\bijN{u^1(0),\dots, u^N(0)}, \dots, \bijN{u^1(n-1),\dots, u^N(n-1)}),
\]
which we denote by $\bijN{u^1,\dots,u^N}$. (Using the function $\bijN{\cdot}$ over $\langle \cdot \rangle$ makes the formulation of our assertion simpler.) 

Given $\alpha_1,\dots,\alpha_N \in \baire$, and by another abuse of the notation, we let $\bijN{\alpha_1,\dots,\alpha_N}$ be the unique $\alpha \in \baire$ for which 
\[
\alpha \upharpoonright t 
= \bijN{\alpha_1\upharpoonright t, \dots, \alpha_N \upharpoonright t}
= (\bijN{\alpha_1(0), \dots, \alpha_N(0)}, \dots, \bijN{\alpha_1(t-1), \dots, \alpha_N(t-1)})
\]
for all $t$. From the surjectiveness of $\bijN{\cdot}$ it follows that for every $u \in \om^{<\om}$ there are $u^1,\dots, u^N \in \om^{<\om}$ such that $u = \bijN{u^1,\dots,u^N}$. Similarly every $\alpha \in \baire$ is of the form $\bijN{\alpha_1,\dots,\alpha_N}$ for some $\alpha_1, \dots, \alpha_N \in \baire$.

It is not hard to verify that $N_u$ is the image of $N_{u^1} \times \dots \times N_{u^N}$ under $\bijN{\cdot}$, where $u = \bijN{u^1,\dots,u^N}$. We denote the latter image by $\bijN{N_{u^1} \times \dots \times N_{u^N}}$.

Finally we define the function $\sigma_N: \baire \to \R^N:$
\[
\alpha = \bijN{\alpha_1,\dots,\alpha_N} \mapsto (\sigma(\alpha_1),\dots,\sigma(\alpha_N)).
\]
Clearly $\sigma_N$ is bijective and recursive. Moreover
\begin{align*}
\sigma_N[N_u] 
=& \ \sigma_N[\bijN{N_{u^1} \times \dots \times N_{u^N}}]\\
=& \ \sigma[N_{u^1}] \times \dots \times \sigma[N_{u^N}]\\
=& \ I_{u^1} \times \dots I_{u^N}.
\end{align*}
The preceding $u^1, \dots, u^N$ can be retrieved recursively from $u$, so that for all $i=1,\dots, N$ the relations
\begin{align*}
P^i(t,s) \iff& \ \text{$\pq{s}$ is the minimum of $I_{u^i}$\;, where $\dec{t} = \bijN{u^1,\dots,u^i,\dots, u^N} \neq \emptyset$}\\
Q^i(t,s) \iff& \ \text{$\pq{s}$ is the supremum of $I_{u^i}$\;, where $\dec{t} = \bijN{u^1,\dots,u^i,\dots, u^N} \neq \emptyset$}
\end{align*}
are recursive. To finish the proof define
\begin{align*}
p_i':& \ \om \to \om : t \mapsto \text{the least $s$ for which} \ P^i_1(t,s)\\
q_i':& \ \om \to \om : t \mapsto \text{the least $s$ for which} \ P^i_2(t,s)\\
p_i:& \ \om \to \Q: p_i(t) = \pq{p_i'(t)}, \quad q_i: \ \om \to \Q: q_i(t) = \pq{q_i'(t)}
\end{align*}
for all $i = 1,\dots,N$.
\end{proof}

\begin{corollary}
\label{corollary real numbers normal Souslin code}
For every natural number $N \geq 1$ there is some $\ep \in \del$ and a $\Sigma^0_1(\ep)$-recursive function $\strans^N_\ast: \baire \to \baire$ such that for all $\alpha \in \baire$ the $\del(\alpha)$ point $\strans^N_\ast(\alpha)$ is a normal Souslin code for $\scode^\ca{X}_1(\alpha)$. 
\end{corollary}

\begin{proof}
We consider the functions $\pi_N: \baire \to \R^N$ and $q_1,p_1,\dots,p_N,q_N: \om \to \R$ as in the statement of Lemma \ref{lemma surjection to Rn convex}, and we notice that for all $x \in \R^N$ we have
\begin{align*}
x \in \overline{f[N(\baire,t)]} 
\iff& \ x \in [p_1(t),q_1(t)] \times \dots \times [p_N(t),q_N(t)]\\
\iff& \ (\forall k)(\exists s_1,\dots,s_N)(\forall i =1,\dots,N)\\
& \ \hspace*{5mm}[p_i(t) < \pq{s_i} < q_i(t) \ \& \ |x_i - \pq{s_i}| < 2^{-k}]
\end{align*}
for all $t \in \om$. The latter relation is clearly $\Pi^0_2$. Hence the hypothesis of Proposition \ref{proposition transition to normal Souslin codes} is satisfied and we are done.
\end{proof}

\subsection{Codes for some open convex sets}

\label{subsection codes for some open convex sets}

As we claimed above it is possible to derive in a \del-way the codes of sets of the form $V_u: = \set{x \in \R^N}{d(x, H(Q_u \cap [-m,m]^N)) <  2^{-n}}$ out of a code of a given Souslin scheme $(Q_u)_{u \in \om^{<\om}}$. There are however two ways here to interpret the term ``code for $V_u$". The first one is to use the open codes, since the sets $V_u$ are evidently open subsets of $\R^N$. The other way reflects the definition of the convexly generated sets. More specifically, one can write the open set $V_u$ as the increasing union of a sequence of compact convex sets $(F_{u,i})_{i \in \om}$. We would then use a closed code $\beta_{u,i}$ for each $F_{u,i}$ to encode $V_u$. The term ``\del-way" means that the function $(u,i) \mapsto \beta_{u,i}$ is $\del$-recursive.\footnote{This method essentially describes the basis step of the \emph{convexly generated codes} that we will introduce in the sequel.} The next proposition settles our preceding claim with respect to both of these ways for encoding.

\begin{proposition}
\label{proposition basis of coding}
For every natural number $N \geq 1$ there is an $\ep \in \del$ and $\Sigma^0_5(\ep)$-measurable functions $f,g: \om^2 \times \baire \to \baire$ satisfying for all $(n,m,\alpha)$ with $\fcode^{\R^N}(\alpha) \cap [-m,m^N] \neq \emptyset$ the following:
\begin{list}{}{}
\item[\tu{(}a\tu{)}] $f(n,m,\alpha)$ is an open code of the set
\[
\set{x \in \R^N}{d\left(x,H\left(\fcode^{\R^N}(\alpha) \cap [-m,m^N]\right)\right) < 2^{-n}}.
\]
\item[\tu{(}b\tu{)}] It holds $g(n,m,\alpha)(0) = 1$, and for all $i \in \om$ we have that $\rfn{g(n,m,\alpha)^\ast}(i)$ is defined and is a closed code for a convex compact subset of $\ocode^{\R^N}(f(n,m,\alpha))$ with
\[
\ocode^{\R^N}(f(n,m,\alpha)) = \cup_{i \in \om} \; \fcode^{\R^N}(\rfn{g(n,m,\alpha)^\ast}(i)).
\]
Moreover $\fcode^{\R^N}(\rfn{g(n,m,\alpha)^\ast}(i)) \subseteq \fcode^{\R^N}(\rfn{g(n,m,\alpha)^\ast}(i+1))$ for all $i \in \om$.
\end{list}
\end{proposition}   

It is worth commenting on the transition from the function $f$ to the function $g$ in the preceding proposition. The idea is to use the fact that every section $F^\alpha: \om \to \baire: i \mapsto F(i,\alpha)$ of a $\Sigma^0_n$-recursive function $F: \om \times \baire \to \baire$ is of the form $\rfn{\gamma(\alpha)}^{\om,\baire}$ for some $\Sigma^0_n$-recursive function $\alpha \mapsto \gamma(\alpha)$. The latter is essentially a consequence of the fact that every $\Sigma^0_n(\alpha)$ subset of the naturals is recursive in the $(n-1)$-th Turing jump of $\alpha$, uniformly on $\alpha$. We will give more details in the sequel.\smallskip

The proof of Proposition \ref{proposition basis of coding} requires a detailed inspection of some classical facts, which include among other things  the result that the convex hull of a compact set is compact (Caratheodory Theorem) and the fact that every open set is convexly generated. To do this we employ a series of Lemmata.

\begin{lemma}
\label{lemma code of continuous image of compact}
Let $\ca{X}$, $\ca{Y}$ be recursive Polish spaces and let $\pi : \ca{X} \to \ca{Y}$ be a continuous $\del$-recursive function. We consider the function $u(\pi) \equiv u: \baire \to \cantor$ defined by
\[
u(\alpha)(s) = 1 \iff \pi^{-1}[N(\ca{Y},s)] \cap \fcode^\ca{X}(\alpha) = \emptyset
\]
where $(s,\alpha) \in \om \times \cantor$. Then $u$ has the following properties:
\begin{list}{}{}
\item[\tu{(}1\tu{)}] If the closed set $\fcode^{\ca{X}}(\alpha)$ is compact \tu{(}including the case of the empty set\tu{)} then $u(\alpha)$ is a closed code for $\pi[\fcode^{\ca{X}}(\alpha)]$, \ie $$\fcode^{\ca{Y}}(u(\alpha)) = \pi[\fcode^{\ca{X}}(\alpha)]; \quad \text{and}$$
\item[\tu{(}2\tu{)}] the function $u$ is computed by a $\sig$ and a $\pii$ relation on $\set{\alpha}{\fcode^\ca{X}(\alpha) \ \text{is compact}}$, \ie there are relations $P, Q \subseteq \baire \times \omega$ in $\sig$ and $\pii$ respectively such that whenever $\fcode^{\ca{X}}(\alpha)$ is compact it holds
\[
u(\alpha) \in N(\baire,s) \iff P(\alpha,s) \iff Q(\alpha,s).
\]
\end{list}
If moreover the pre-image of every closed ball $\set{y \in \ca{Y}}{d_\ca{Y}(y,y_0) \leq r}$ under $\pi$ is a compact subset of \ca{X} then \tu{(}2\tu{)} is strengthened to the statement that $u$ is $\Sigma^0_3(\ep)$-recursive for some $\ep \in \del$.\footnote{Although $\del$-recursiveness is sufficient for our purposes, we state this and the subsequent results in their full strength.}
\end{lemma}

\begin{proof}
We fix some compact set $K \subseteq  \ca{X}$ for the discussion. For all \n, let
\[
I= \set{s \in \om}{\pi^{-1}[N(\ca{Y},s)] \cap K = \emptyset}.
\]
We claim that $\ca{Y} \setminus \pi[K] = \cup_{s \in I} N(\ca{Y},s)$. To see this consider some $y \in \ca{Y} \setminus \pi[K]$, and assume towards a contradiction that for all $s \in I$ we have that $y \not \in N(\ca{Y},s)$. This means that for all $s \in \om$ with $y \in N(\ca{Y},s)$ it holds $\pi^{-1}[N(\ca{Y},s)] \cap K \neq \emptyset$. There is therefore a sequence $(x_n)_{\n}$ in $K$ such that $\pi(x_n) \to y$. Without loss of generality we may assume that the sequence $(x_n)_{\n}$ converges to some $x \in K$. From the continuity of $\pi$ we have that $\pi(x_n) \to \pi(x)$, and hence $y = \pi(x) \in \pi[K]$, a contradiction. For the converse inclusion if $y$ is a member of $N(\ca{Y},s)$ for some $s \in I$, and (towards a contradiction) $y = \pi(x)$ for some $x \in K$, then $x \in \pi^{-1}[N(\ca{Y},s)] \cap K$, which implies that $s \not \in I$.

From the preceding it follows that the set $I$ above gives a closed code for the set $\pi[K]$. Define the set $J \subseteq \baire \times \om$ by
\[
J(\alpha,s) \iff \pi^{-1}[N(\ca{Y},s)] \cap \fcode^\ca{X}(\alpha) = \emptyset,
\]
so that from above $\ca{Y} \setminus \pi[\fcode^{\ca{X}}(\alpha)] = \cup_{s \in J(\alpha)} N(\ca{Y},s)$ whenever $\fcode^\ca{Y}(\alpha)$ is compact, where $J(\alpha)$ is as usual the $\alpha$-section of $J$. Since $u(\alpha)(s) = 1 \iff s \in J(\alpha)$ we have that $\pi[\fcode^{\ca{X}}(\alpha)] = \fcode^\ca{Y}(u(\alpha))$ whenever $\fcode^\ca{Y}(\alpha)$ is compact.

We now check property (2). Define $B \subseteq \om^2 \times \ca{X}$ by
\[
B(s,k,x) \iff d_\ca{Y}(\pi(x),r^\ca{Y}_{(s)_0}) \leq \pq{{(s)_1}}.
\]
Clearly $B$ is $\del$ and closed. Let $\fcode^{\ca{X}}(\alpha)$ be compact. Then for all $s$ we have
\begin{align*}
u(\alpha)(s) = 0 
\iff& \ \pi^{-1}[N(\ca{Y},s)] \cap \fcode^{\ca{X}}(\alpha) \neq \emptyset \\
\iff& \ (\exists x)[ x \in \fcode^\ca{X}(\alpha) \ \& \ d_\ca{Y}(\pi(x),r^\ca{Y}_{(s)_0}) < \pq{{(s)_1}}]\\
\iff& \ (\exists k)(\exists x)[x \in \fcode^\ca{X}(\alpha) \ \& \ d_\ca{Y}(\pi(x),r^\ca{Y}_{(s)_0}) \leq \pq{s} - 2^{-k}]\\
\iff & \ (\exists k)(\exists x)[x \in \fcode^\ca{X}(\alpha) \cap B(s,k)]\\
\iff& \ (\exists k)(\exists x \in \del(\alpha))[x \in \fcode^\ca{X}(\alpha) \cap B(s,k)];
\end{align*}
where the latter of the preceding equivalences follows from the fact that the set $\fcode^\ca{X}(\alpha) \cap B(s,k)$ is a $\del(\alpha)$ closed subset of the compact set $\fcode^\ca{X}(\alpha)$, and so from \cite[4F.15]{yiannis_dst} it contains a $\del(\alpha)$-point whenever it is not empty. The required sets $P$, $Q$ are finite Boolean combinations of the right-hand side of the latter two equivalences.

Now assume for the remaining of this proof that the pre-image of every closed ball in \ca{Y} under $u$ is compact. Then the preceding equivalences are valid for all $\alpha \in \baire$, since the set $\fcode^\ca{X}(\alpha) \cap B(s,k)$ is a $\del(\alpha)$ closed subset of the compact set $B(s,k)$, and the same argument applies. This shows that $u$ is $\del$-recursive.

Finally we show that $u$ is also $\bolds^0_3$-measurable. From above we have for all $\alpha, s$ that
\begin{align*}
\pi^{-1}[N(\ca{Y},s)] \cap \fcode^{\ca{X}}(\alpha) = \emptyset  
\iff& \
(\forall k)(\forall x)[d_\ca{Y}(\pi(x),r^\ca{Y}_{(s)_0}) \leq \pq{s} - 2^{-k} \ \longrightarrow \ x \not \in \fcode^\ca{X}(\alpha)] \\
\iff & \
(\forall k)(\forall x)[x \in B(s,k) \ \longrightarrow \ (\exists n)[x \in N(\ca{X},\alpha(n))]]\\
\iff& \
(\forall k)(\exists n_1,\dots,n_{m(k)})[B(s,k) \subseteq \cup_{i=1}^{m(k)} N(\ca{X},\alpha(n_i))],
\end{align*}
where in the last equivalence we used the compactness of $B(s,k)$. It is clear that for all $w = (n_1,\dots,n_m) \in \om^{< \om}$ and all $s,k \in \om$ the set
\[
C_{w,s,k} := \set{\alpha \in \baire}{B(s,k) \subseteq \cup_{i=1}^{m} N(\ca{X},\alpha(n_i))}
\]
is a clopen subset of $\baire$. Hence for all $s \in \om$ the pre-image 
\[
u^{-1}[\set{\beta \in \baire}{\beta(s) = 1}] = \cap_{k \in \om}\cup_{w \in \om^{<\om}\setminus \{\emptyset\}}\; C_{w,s,k}
\]
is a $\boldp^0_2$ subset of $\baire$. It follows that $u$ is $\bolds^0_3$-measurable. 

To show that $u$ is $\Sigma^0_3(\ep)$-recursive for some $\ep \in \del$ we apply the \emph{Louveau Separation} \cite{louveau_a_separation_theorem_for_sigma_sets} to the neighborhood diagram relation $R^u \subseteq \baire \times \om$ given by 
\[
R^u(\alpha,s) \iff u(\alpha) \in N(\baire,s).
\]
The latter is $\del$ and $\bolds^0_3$ since $u$ is $\del$-recursive and $\bolds^0_3$-measurable, and so from the Louveau Separation $R^u$ is $\Sigma^0_3(\ep)$ for some $\ep \in \del$.
\end{proof}

\begin{lemma}
\label{lemma transition from closed code to code of convex hull}
For all $N \geq 1$ there is some $\ep \in \baire$ and a $\Sigma^0_3(\ep)$-recursive function $v: \om \times \baire \to \baire$ such that $v(m,\alpha)$ is a closed code for the convex hull of the compact set $\fcode^{\R^N}(\alpha) \cap [-m,m]^N$ \tu{(}even if the latter is the empty set\tu{)}, \ie
\[
\fcode^{\R^N}(v(m,\alpha)) = H\left(\fcode^{\R^N}(\alpha) \cap [-m,m]^N\right),
\] 
for all $(m,\alpha)$.
\end{lemma}

\begin{proof}
We consider the recursively presented metric spaces
\begin{align*}
\ca{X} =  [0,1]^{N+1} \times 
\underbrace{\R^N \times \dots \times \R^N}_{(N+1)\text{-times}}\;,
\quad
\ca{Y} = \ \R^N\;,
\end{align*}
and the recursive function
\[
\pi: \ca{X} \to \ca{Y}: (a_1,\dots,a_{N+1},x_1,\dots,x_{N+1}) \mapsto \sum_{i=1}^{N+1} a_i \cdot x_i.
\]
Notice the from the Caratheodory's Theorem the convex hull $H(K)$ of a compact set $K \subseteq \R^N$ is the image of $[0,1]^{N+1} \times K^{N+1}$ under $\pi$ (which also shows that $H(K)$ is compact).

We consider the function $u(\pi) \equiv u$, which corresponds to $\pi$ as in the statement of Lemma \ref{lemma code of continuous image of compact}. We will show the following:
\begin{list}{}{}
\item[(a)] There is a recursive function $v_0: \om \times \baire \to \baire$ such that $v_0(m,\alpha)$ is an $\ca{X}$-closed code for $[0,1]^{N+1} \times \left(\fcode^{\R^N}(\alpha) \cap [-m,m]^N\right)^{N+1}$ for all $(m,\alpha)$.
\item[(b)] The preceding function $u$ is $\Sigma^0_3(\ep)$-recursive for some $\ep \in \del$.
\footnote{It is actually easy to verify the weaker assertion that the function $v = u \circ v_0$ is $\del$-recursive using (a): the set $\fcode^{\ca{X}}(v_0(m,\alpha))$ is compact for all $m,\alpha$, and so from (2) of Lemma \ref{lemma code of continuous image of compact} the function $u$ is $\del$-recursive. Since $v_0$ is recursive, it follows that $v$ is $\del$-recursive.}
\end{list}
If we show these two assertions we will be done, for then we can take the $\Sigma^0_3(\ep)$-recursive function $v = u \circ v_0$, and using (1) of Lemma \ref{lemma code of continuous image of compact} we will have
\begin{align*}
\fcode^{\R^N}(v(m,\alpha))
=& \ \fcode^{\R^N}(u (v_0(m,\alpha)))\\ 
=& \ \pi[\fcode^{\ca{X}}(v_0(m,\alpha))]\\
=& \ \pi \left[[0,1]^{N+1} \times \left(\fcode^{\R^N}(\alpha) \cap [-m,m]^N\right)^{N+1}\right]\\
=& \ H\left(\fcode^{\R^N}(\alpha) \cap [-m,m]^N\right).
\end{align*}

The assertions (a) and (b) are proved in a series of claims. Before we proceed we find it useful to outline the main steps. Assertion (a) is straightforward, we start from a code for $[-m,m]$, pass to a code for $[-m,m]^N$ and then to a code for $\fcode^{\R^N}(\alpha) \cap [-m,m]^N$ uniformly on $m,\alpha$. This requires a long list of computations, which we give below in some (but not full) detail. 

Regarding (b), we consider for all $m \in \om$ the recursively presented metric space
\[
\ca{X}_m = \ca{X} \cap \left([0,1]^{N+1} \times \underbrace{[-m,m]^N \times \dots \times [-m,m]^N}_{(N+1)\text{-times}}\right).
\]
It is clear that $\ca{X} = \cup_{m \in \om} \ca{X}_m$. Let $\pi_m$ be the restriction of $\pi$ on \ca{X} and  $u(\pi_m) \equiv u_m$ be the function, which corresponds to $\pi_m$ as in the statement of Lemma \ref{lemma code of continuous image of compact}.

We fix some $m \in \om$ for the discussion. Since $\ca{X}_m$ is compact and $\pi_m$ is continuous, the pre-image of every closed ball of $\R^N$ under $\pi_m$ is compact. Notice that from Lemma \ref{lemma code of continuous image of compact} the function $u_m$ is $\Sigma^0_3(\gamma_m)$-recursive for some $\gamma_m \in \del$. We need however something stronger for (b). As it is evident from the proof of the latter lemma the set $L: = \set{(\alpha,t)}{ u(\alpha)(t) = 0}]$ is $\bolds^0_2$ and $\del$. Hence from the Louveau Separation there is some $\delta_m \in \del$ such that $L$ is a $\Sigma^0_2(\delta_m)$ set.

From the properties of our universal systems there is some $\delta_m$-recursive point $\ep_m \in \baire$ such that for all $\alpha, t$ it holds
\[
u_m(\alpha)(t) = 0 \iff \scode^{\baire \times \om}_2(\ep_m,\alpha,t).
\]
Notice that each $\ep_m$ is a $\del$-point. 

Our goal is to replace all the preceding $\ep_m$'s with $(\ep)_m$ for some $\ep \in \del$. To do this we need to know that the condition ``$u_m(\alpha)(t) = 0$" is verified $\del$-uniformly on $m,\alpha,t$, \ie that the set $Q : = \set{(m,t,\alpha) \in \om^2 \times \baire}{u_m(\alpha)(t) = 0}$ is $\del$. (This is Claim 6 below.) If we achieve this, then we define $P \subseteq \om \times \baire$ by
\[
P(m,\beta) \iff (\forall \alpha,t)[Q(m,t,\alpha) \longleftrightarrow \scode^{\baire \times \om}_2(\beta,\alpha,t)].
\]
Clearly $P$ is $\pii$, and for all $m \in \baire$ there is $\beta \in \del$ (namely $\ep_m$) such that $P(m,\beta)$. From the $\Delta$-Selection Principle \cite[4D.6]{yiannis_dst} there is a $\del$-recursive function $U: \om \to \baire$ such that $P(m,U(m))$ for all $m$. We then take an $\ep \in \del$ with $(\ep)_m = U(m)$ for all $m$, and we have
\[
u_m(\alpha)(t) = 0 \iff Q(m,t,\alpha) \iff \scode^{\baire \times \om}_2((\ep)_m,\alpha,t)
\]
for all $\alpha,t, m$, ( in particular every function $u_m$ is $\Sigma^0_3(\ep)$-recursive). From this we can show that $u$ is also $\Sigma^0_3(\ep)$-recursive: first we consider a recursive function $s \in \om \mapsto t(s) \in \om$ such that $N(\baire,t(s)) = \set{\beta \in \baire}{\beta(s) = 0}$, and using the recursive function $f$ of Claim 5 (see below) we will have:
\begin{align*}
u(\alpha)(s) = 0 
\iff& \ \pi^{-1}[N(\ca{Y},s)] \cap \fcode^\ca{X}(\alpha) \neq \emptyset\\
\iff& \ (\exists m)[\pi^{-1}[N(\ca{Y},s)] \cap \fcode^\ca{X}(\alpha) \cap \ca{X}_m \neq \emptyset]\\
\iff& \ (\exists m)[\pi_m^{-1}[N(\ca{Y},s)] \cap \fcode^\ca{X}(\alpha) \cap \ca{X}_m \neq \emptyset]\\
\iff& \ (\exists m)[\pi_m^{-1}[N(\ca{Y},s)] \cap \fcode^{\ca{X}_m}(f(m,\alpha)) \neq \emptyset]\\
\iff& \ (\exists m)[u_m(f(m,\alpha))(s) =0]\\
\iff& \ (\exists m)\scode^{\baire \times \om}_2((\ep)_m,f(m,\alpha),t(s)).
\end{align*}
Hence the condition $``u(\alpha) \in N(\baire,k)$" can be expressed as a finite boolean combination of $\Sigma^0_2(\ep)$ sets, and therefore $u$ is $\Sigma^0_3(\ep)$-recursive. Hence Claim 6 is all we need to establish assertion (b).\smallskip

Now we proceed to the claims.\smallskip

\emph{Claim 1.} There is a recursive function $v_{1}: \om \times \baire \to \baire$ such that whenever $\alpha \in \baire$ is such that $(\alpha)_m$ is a code of the closed set $[-m,m]$ in $\R$, then $v_1(m,\alpha)$ is a closed code for $[-m,m]^N$ in $\R^N$.

\emph{Proof of the claim.} Fix a recursive function $n \mapsto c(n)$ such that $(-n,n) = N(\R,c(n))$ for all \n. Using that the finite products of basic neighborhoods is a finite product witnessed by the recursive function $\fprod$ (see (\ref{equation product of neighborhoods}) in the Introduction), it follows that
\begin{align*}
\R^N \setminus [-m,m]^N 
=& \ \bigcup_{i=0}^{N-1} \left(\R \times \dots \times \underbrace{\left(\R \setminus [-m,m]\right)}_{i-\text{th place}} \times \dots \times \R\right)\\
=& \ \bigcup_{i=0}^{N-1}\bigcup_{n \in \om}(-n,n) \times \dots \times (\R \setminus [-m,m]) \times \dots \times (-n,n)\\
=& \ \bigcup_{i=0}^{N-1}\bigcup_{n,j \in \om}N(\R,c(n)) \times \dots \times N(\R,(\alpha)_m(j)) \times \dots \times N(\R,c(n))\\
=& \ \bigcup_{i=0}^{N-1}\bigcup_{n,j \in \om}\bigcup_{k \in \om} N(\R^N, \fprod(\langle c(n), \dots, \underbrace{(\alpha)_m(j)}_{i-\text{th place}}, \dots, c(n)\rangle)(k))\\
=& \ \bigcup_{s \in \om}N(\R^N,v_1(m,\alpha)(s)),
\end{align*}
where
\[
v_1(m,\alpha)(s) =
\begin{cases}
\fprod(\langle c(n), \dots, \underbrace{(\alpha)_m(j)}_{i-\text{th place}}, \dots, c(n)\rangle)(k), & \ \text{if} \ s= \langle i,n,j,k \rangle \ \text{and} \ 0 \leq i \leq N-1\\
0, & \ \text{else}.
\end{cases}
\]
Clearly $v_1$ is recursive and $v_1(m,\alpha)$ is a closed code for $[-m,m]^N$ in $\R^N$.\smallskip

\emph{Claim 2.} There is a recursive function $v_2: \om \times \baire \to \baire$ such that
\[
\fcode^{\R^N}(\alpha) \cap [-m,m]^N = \fcode^{\R^N}(v_2(m,\alpha)) \quad \text{for all} \ m,\alpha.
\]
\emph{Proof of the claim.} 
It is not hard to find a recursive code for $[-m,m]$ uniformly on $m$, \ie there is some recursive $\delta \in \baire$ such that $(\delta)_m$ is a closed code for $[-m,m]$ in $\R$ for all $m \in \om$.

Recall that there is a recursive function $v_{\wedge}: \baire \times \baire \to \baire$ such that
\[
\fcode^{\R^N}(\alpha) \cap \fcode^{\R^N}(\beta) = \fcode^{\R^N}(v_{\wedge}(\alpha,\beta))
\]
for all $\alpha$, $\beta$.  So we take the recursive function
\[
v_2: \om \times \baire \to \baire: (m,\alpha) \mapsto v_{\wedge}(\alpha,v_1(m,\delta)),
\]
where $v_1$ is as in Claim 1. Since $\delta$ is recursive, the function $v_2$ is recursive as well. Moreover
\[
\fcode^{\R^N}(v_1(m,\alpha)) = \fcode^{\R^N}(v_{\wedge}(\alpha,v_1(m,\delta)))= \fcode^{\R^N}(\alpha) \cap \fcode^{\R^N}(v_1(m,\delta)) = \fcode^{\R^N}(\alpha) \cap [-m,m]^N.
\]
\emph{Claim 3.} There is a recursive function $v_3: \baire \to \baire$ such that $v_3(\beta)$ is a closed code for $[0,1]^{N+1} \times \left(\fcode^{\R^N}(\beta)\right)^{N+1}$ in \ca{X} for all $\beta \in \baire$.

The \emph{proof of Claim 3} is similar to the one of Claim 1 and we omit it.\smallskip

\emph{Proof of assertion} (a). We consider the recursive function
\[
v_0: \om \times \baire \to \baire: (m,\alpha) \mapsto v_3(v_2(m,\alpha)),
\]
and we have
\[
\fcode^{\ca{X}}(v_0(m,\alpha)) = [0,1]^{N+1} \times \left(\fcode^{\R^N}(v_2(m,\alpha))\right)^{N+1} = [0,1]^{N+1} \times \left(\fcode^{\R^N}(\alpha) \cap [-m,m]^{N}\right)^{N+1}.
\]
Now we proceed to the claims need for the proof of (b). Notice that the set $\fcode^{\ca{X}_m}(\alpha)$ (which is also a closed subset of \ca{X}) might not the same as $\fcode^\ca{X}(\alpha)$. This is because the $s$-th basic neighborhood of $\ca{X}_m$ is not necessarily the intersection of the $s$-th basic neighborhood of $\ca{X}$ with $\ca{X}_m$. This boils down to the fact that the basic neighborhoods of $\ca{X}_m$ are indexed according to some enumeration of the rationals in the interval $[-m,m]$, and the latter enumeration is clearly different from any enumeration of all rationals. Of course this is just an indexing issue:

\emph{Claim 4.} There are recursive functions $g,h: \om \times \baire \to \baire$ such that 
\[
\fcode^{\ca{X}_m}(\alpha) = \fcode^\ca{X}(g(m,\alpha)) \cap \ca{X}_m = \fcode^\ca{X}(h(m,\alpha))
\]
for all $m,\alpha$.

\emph{Proof of the claim.} By our representation of the intervals of the form $[-m,m]$ there is a recursive function $l: \om^2 \to \om$ such that $\dense{[-m,m]}{i} = \dense{\R}{l(m,i)}, \text{for all} \ m,i$.
We then have
\begin{align*}
N([-m,m],s) 
=& \ \set{x \in [-m,m]}{|x - \dense{[-m,m]}{(s)_0}| < \pq{{(s)_1}}} \\
=& \ \set{x \in \R}{|x - \dense{\R}{l(m,(s)_0)}| < \pq{{(s)_1}}} \cap [-m,m]\\
=& \ N(\R,\langle l(m,(s)_0), (s)_1 \rangle) \cap [-m,m].
\end{align*}
By taking products we can find a recursive function $p': \om^2 \to \om$ such that
\[
N([-m,m]^N,s) = N(\R^N, p'(m,s)) \cap [-m,m]^N, \quad \text{for all} \ m,s;
\]
and by repeating this procedure there is a recursive function $p: \om^2 \to \om$ such that
\begin{align*}
& \ N\left([0,1]^{N+1} \times \underbrace{[-m,m]^{N} \times \dots \times [-m,m]^N}_{(N+1)\text{-times}}\;, \; t\right)\\
=& \ N\left([0,1]^N \times \underbrace{\R^N \times \dots \R^N}_{(N+1)\text{-times}}\;, \; p(m,t)\right) \cap \underbrace{[-m,m]^{N} \times \dots \times [-m,m]^{N}}_{(N+1)\text{-times}},
\end{align*}
in other words
\[
N(\ca{X}_m,t) = N(\ca{X},p(m,t)) \cap \ca{X}_m, \quad \text{for all} \ m,t.
\]
Hence for all $(\vec{a},\vec{x}) \in \ca{X}_m$ we have
\begin{align*}
(\vec{a},\vec{x}) \not \in \fcode^{\ca{X}_m}(\alpha)
\iff& \ (\exists i)[(\vec{a},\vec{x}) \in N(\ca{X}_m,\alpha(i))]\\
\iff& \ (\exists i) [(\vec{a},\vec{x}) \in N(\ca{X},p(m,\alpha(i)))]\\
\iff& \ (\vec{a},\vec{x}) \not \in \fcode^{\ca{X}}(g(m,\alpha)),
\end{align*}
where $g(m,\alpha)(i) = p(m,\alpha(i))$. Thus $ \fcode^{\ca{X}_m}(\alpha) =  \fcode^{\ca{X}}(g(m,\alpha)) \cap \ca{X}_m$ and $g$ is recursive.

To obtain the function $h$ we need to identify a code for $\ca{X}_m$ as a closed subset of \ca{X} uniformly on $m$. This is straightforward. First there is a recursive function $v_4: \om \times \baire \to \baire$ such that whenever $(\beta)_m$ is a closed code for $[-m,m]^N$ in $\R^N$ then $v_4(m,\beta)$ is a closed code for $[0,1]^{N+1} \times \underbrace{[-m,m]^N \times \dots \times [,m,m]^N}_{(N+1)\text{-times}} = \ca{X}_m$ in \ca{X}. (This is proved exactly as in Claim 1.) 

Now, as in the proof of Claim 2, we consider a recursive $\delta \in \baire$ such that $(\delta)_m$ is a closed code for $[-m,m]$ in $\R$ for all $m \in \om$. We define
\[
v_5: \om \to \baire: m \mapsto v_4(v_1(m,\delta)),
\]
where $v_1$ is as in Claim 1. In particular $v_1(m,\delta)$ is a code for the closed set $[-m,m]^N$ in $\R^N$, and hence $v_4(v_1(m,\delta)) = v_5(m)$ is a code for $\ca{X}_m$ in \ca{X}.

Using the preceding functions $g$ and $v_5$ we have
\[
\fcode^{\ca{X}}(g(m,\alpha)) \cap \ca{X}_m = \fcode^{\ca{X}}(g(m,\alpha)) \cap \fcode^{\ca{X}}(v_5(m)) = \fcode^{\ca{X}}(v^\ca{X}_{\wedge}(g(m,\alpha),v_5(m))),
\]
where $v^\ca{X}_{\wedge}$ is as in $v_{\wedge}$ with the space $\ca{X}$ in the place of $\R^N$. So we take the recursive function $h(m,\alpha) = g(m,\alpha),v_5(m))$. This finishes the proof of Claim 4.\smallskip

We also need a uniform way for the converse direction of Claim 4, \ie a way to pass from closed sets of the form $F^\ca{X}(\beta) \cap \ca{X}_m$ to closed codes in $\ca{X}_m$.\smallskip

\emph{Claim 5.} There is a recursive function $f: \om \times \baire \to \baire$ such that
\[
\fcode^{\ca{X}}(\alpha) \cap \ca{X}_m = \fcode^{\ca{X}_m}(f(m,\alpha)), \quad \text{for all} \ m,\alpha.
\]
The \emph{proof of Claim 5} is immediate from the preceding proof: we use the functions $v^\ca{X}_{\wedge}$ and $v_5$ - the only difference is that now we have $\alpha$ instead of $g(m,\alpha)$.\smallskip

\emph{Claim 6.} The relation $Q \subseteq \om^2 \times \baire$ defined by
\[
Q(m,t,\alpha) \iff u_m(\alpha) \in N(\baire,t)
\]
is \del.

\emph{Proof of claim.} Using the function $h$ of Claim 4 we compute
\begin{align*}
u_m(\alpha)(s) = 0 
\iff& \ \pi_m^{-1}[N(\ca{Y},s)] \cap \fcode^{\ca{X}_m}(\alpha) \neq \emptyset\\
\iff& \ \pi^{-1}[N(\ca{Y},s)] \cap \fcode^{\ca{X}_m}(\alpha) \neq \emptyset\\
\iff& \ \pi^{-1}[N(\ca{Y},s)] \cap \fcode^{\ca{X}}(h(m,\alpha)) \neq \emptyset\\
\iff& \ u(h(m,\alpha))(s) = 0.
\end{align*}
Since the set $\fcode^\ca{X}(h(m,\alpha)) = \fcode^{\ca{X}_m}(\alpha)$ is a compact subset of $\ca{X}$, it follows from (2) of Lemma \ref{lemma code of continuous image of compact} that the function $(m,\alpha) \mapsto (u \circ h)(m,\alpha)$ is \del-recursive. Hence the condition ``$u_m(\alpha)(s) = 0 $" is \del, and therefore $Q$ is \del \ as well.\smallskip

As shown above assertion (b) is a consequence of Claim 6. This finishes the proof of the lemma. 
\end{proof}

\begin{lemma}
\label{lemma transition from closed code to distance from convex hull}
Suppose that $N \geq 1$ and 
\[
\ca{A} = \set{(m,\alpha) \in \om \times \baire }{\fcode^{\R^N}(\alpha) \cap [-m,m]^N \neq \emptyset}.
\]
Then for some $\ep \in \del$ we have the following:
\begin{list}{}{}
\item[\tu{(}a\tu{)}] The set $\ca{A}$ is $\Pi^0_1(\ep)$.
\item[\tu{(}b\tu{)}] The function
\[
f_d:  \ca{A} \times \R^N \to \R: (m,\alpha,x) \mapsto d\left(x,H\left(\fcode^{\R^N}(\alpha) \cap [-m,m]^N\right)\right)
\]
\tu{(}where $d$ is the distance in $\R^N$ with respect to the maximum norm\tu{)} is computed by a is $\Sigma^0_4(\ep)$ set.
\end{list}
\end{lemma}

\begin{proof}
It suffices to show that the set $\ca{A}$ is closed \del, and that the function $f_d$ is computed by a set which is $\bolds^0_4$ and \del. By an application of the Louveau Separation there are $\ep_1, \ep_2 \in \del$ such that $\ca{A}$ is $\Pi^0_1(\ep_1)$ and $f_d$ is computed by a $\bolds^0_4(\ep_2)$ set. By taking $\ep = (\ep_1(0),\ep_2(0),\ep_1(1),\ep_2(1),\dots)$ we obtain the result.

First we show that the set $\ca{A}$ is $\del$: it is evident that $\ca{A}$ is $\sig$ and by using again that non-empty $\del(\alpha)$ compact sets contain a $\del(\alpha)$ point, we see that
\[
(m,\alpha) \in \ca{A} \iff (\exists y \in \del(\alpha))[\fcode^{\R^N}(\alpha,y) \ \& \ y \in [-m,m]^N].
\] 
The right-hand-side of the preceding equivalence defines of course a $\pii$ relation. 

Next we show that $\ca{A}$ is closed. We fix some $m \in \om$. Clearly it is enough to show that the section $\ca{A}_m$ is closed. Since $[-m,m]^N$ is compact we have that
\begin{align*}
\fcode^{\R^N}(\alpha) \cap [-m,m]^N \neq \emptyset 
\iff& \ \left(\cap_{\n} \left(\R^N \setminus N(\R^N,\alpha(n))\right)\right) \cap [-m,m]^N \neq \emptyset\\
\iff& \ (\forall n)[ \cap_{i \leq n} \left((\R^N \setminus N(\R^N,\alpha(i))\right) \cap [-m,m]^N \neq \emptyset].
\end{align*}
For all \n, the set $C^m_n : = \set{\alpha}{\cap_{i \leq n} \left((\R^N \setminus N(\R^N,\alpha(i))\right) \cap [-m,m]^N \neq \emptyset}$ is easily clopen, and so $\ca{A}_m = \cap_n C^m_n$ is closed.\smallskip

Now we deal with the function $f_d$. We consider the sets 
\[
\ca{B}_0 = \set{\beta \in \baire}{\fcode^{\R^N}(\beta) \neq \emptyset} \times \R^N, \quad
\ca{B} = \ca{B}_0 \ \cap \ \left(\set{\beta \in \baire}{\fcode^{\R^N}(\beta) \ \text{is compact}} \times \R^N\right)
\]
and the function
\[
g: \ca{B}_0 \to \R: (\beta, x) \mapsto d(x,\fcode^{\R^N}(\beta)).
\]
First we check that $g$ is computed by a $\sig$ and a $\pii$ relation on $\ca{B}$. To see this we compute
\begin{align*}
d(x,\fcode^{\R^N}(\beta)) < \pq{s} 
\iff& \ (\exists y)[ y \in \fcode^{\R^N}(\beta) \ \& \ d(x,y) < \pq{s}]\\
 \iff& \ (\exists y)(\forall n)[y \not \in N(\R^N,n) \ \& \ d(x,y) < \pq{s}].
\end{align*}
If $\fcode^{\R^N}(\beta)$ is compact then the preceding $y$ can be chosen to be $\del(\beta)$. By taking suitable finite Boolean combinations we obtain $\sig$ and $\pii$ relations $R^g_{\Sigma}$, $R^g_{\Pi}$ respectively such that
\[
d(x,\fcode^{\R^N}(\beta)) \in N(\R,s) \iff R^g_{\Sigma}(\beta,x,s) \iff R^g_{\Pi}(\beta,x,s),
\]
for all $(\beta,x,s) \in \ca{B} \times \om$.

Next we show that that the function $g \upharpoonright \ca{B}$ is $\bolds^0_2$-measurable. In fact we show that for all $a \in \R$ and all $(\beta,x) \in \ca{B}$ with $g(\beta,x) > a$ there is $p > 0$ such that $g[B_{\R^N \times \baire}((\beta,x),p) \cap \ca{B}_0]  \subseteq (a,\infty)$. This implies that the set $g^{-1}[(a,\infty)] \cap \ca{B}$ is open in $\ca{B}$, and so the restriction $g \upharpoonright \ca{B}$ of $g$ on \ca{B} is $\bolds^0_2$-measurable $-$ in fact lower semi-continuous.

Fix $(\beta,x) \in \ca{B}$ and $a \in \R$ with $g(\beta,x) > a$. Since the set $\fcode^{\R^N}(\beta)$ is compact non-empty, there is some $y$ in the latter set such that $g(\beta,x) = d(x,\fcode^{\R^N}(\beta)) = \normR{x-y}$. Fix some $r > 0$ such that $a + 2r < \normR{x-y}_{\infty}$.

The set $C:= \overline{B_{\R^N}(x,a+2r)} = \set{w \in \R^N}{\normR{x-w}_{\infty} \leq a + 2r}$ is compact and moreover $C \cap \fcode^{\R^N}(\beta) = \emptyset$. (Else there would be some $y' \in \fcode^{\R^N}(\beta)$ such that $\normR{x - y'}_\infty \leq a +2r < \normR{x-y}_{\infty} = d(x,\fcode^{\R^N}(\beta))$, a contradiction.) Hence $C \subseteq \cup_{\n} N(\R^N,\beta(n))$, and from the compactness of $C$ there is $n_0$ such that $C \subseteq \cup_{i=0}^{n_0}N(\R^N,\beta(i))$.

We consider the open neighborhood $V: = B_{\R^N}(x,r) \times \set{\gamma \in \baire}{(\forall i \leq n_0)[\gamma(i) = \beta(i)]}$ (which can be easily brought to the form $B_{\R^N \times \baire}((\beta,x),p)$ for some $p > 0$). Given $(z,\gamma) \in V \cap \ca{B}_0$ (not necessarily in $\ca{B}$) we have that $B_{\R^N}(z,a+r) \subseteq B_{\R^N}(x,a+2r) \subseteq C$, and so $B_{\R^N}(z,a+r) \subseteq \cup_{i=0}^{n_0} N(\R^N,\gamma(i))$. If it were $g(z,\gamma) < a+r$ then there would be some $y' \in \fcode^{\R^N}(\gamma)$ such that $\normR{z-y'}_{\infty} < a+r$, and so $y$ would be a member of $\cup_{i=0}^{n_0} N(\R^N,\gamma(i)) \subseteq \R^N \setminus \fcode^{\R^N}(\gamma)$, a contradiction. Hence $g(z,\gamma) \geq a+r > a$.
 
Now we consider the function $v: \om \times \baire \to \om$ of Lemma \ref{lemma transition from closed code to code of convex hull}. For all $(m,\alpha,x) \in \ca{A} \times \R^N$ we have
\begin{align*}
f_d(m,\alpha,x) = d(x,H(\fcode^{\R^N}(\alpha) \cap [-m,m]^N)) = d(x, \fcode^{\R^N}(v(m,\alpha))) = g(v(m,\alpha),x).
\end{align*}
It follows that the function $f_d$ is $\bolds^0_4$-measurable as the composition of a $\bolds^0_3$-measurable with a $\bolds^0_2$ one. Finally we show that $f_d$ is computed by a set in the class $\bolds^0_4 \cap \del$.

For all $(m,\alpha,x) \in \ca{A}\times \R^N$ we have
\begin{align*}
f_d(m,\alpha,x)  \in N(\R,s)
\iff& \ g(v(m,\alpha),x)  \in N(\R,s)\\
\iff& \ (\exists \beta)[\beta = v(m,\alpha) \ \& \ g(\beta,x)  \in N(\R,s)]\\
\iff& \ (\exists \beta)[\beta = v(m,\alpha) \ \&  R^g_{\Sigma}(\beta,x)  \in N(\R,s)]\\
\iff& \ (\forall \beta)[\beta = v(m,\alpha) \ \vee \ R^g_{\Pi}(\beta,x)  \in N(\R,s)].
\end{align*}
Hence the set
\begin{align*}
R^{f_d} 
=& \ (\ca{A} \times \R^N \times \om) \cap \set{(m,\alpha,x,s)}{(\exists \beta)[\beta = v(m,\alpha) \ \& \ R^g_{\Sigma}(\beta,x)]}\\
=& \ (\ca{A} \times \R^N \times \om) \cap \set{(m,\alpha,x,s)}{(\forall \beta)[\beta = v(m,\alpha) \ \vee \ R^g_{\Pi}(\beta,x)]}\\
=& \ (\ca{A} \times \R^N \times \om) \cap \set{(m,\alpha,x,s)}{f_d(m,\alpha,x)  \in N(\R,s)}
\end{align*}
is $\del$ and computes $f_d$. Moreover $R^{f_d}$ is $\bolds^0_4$ since $\ca{A}$ is closed and $f_d$ is $\bolds^0_4$-measurable. 
\end{proof}\bigskip

Now we are ready for the \emph{proof to Proposition \ref{proposition basis of coding}}. First we remark the following.

\emph{Claim.} Given a separable metric space $(X,d)$, a dense sequence $(x_i)_{i \in \om}$ in $X$ and a non-empty $A \subseteq X$, then for all $x \in X$ and all $p > 0$ we have
\begin{align*}
d(x,A) < p 
\iff& \ (\exists s,i)[d(x_i,A) < p-\pq{s} \ \& \ d(x,x_i) < \pq{s}]\\
\iff& \ (\exists s,i)[d(x_i,A) < p-\pq{s} \ \& \ d(x,x_i) \leq \pq{s}].
\end{align*}
Hence
\[
\set{x \in X}{d(x,A) < p} = \cup_{(s,i) \in I(p,A)}\; B_X(x_i,\pq{s}) = \cup_{(s,i) \in I(p,A)}\; \overline{B_X(x_i,\pq{s})},
\]
where $I(p,A)$ is the set of all $(i,s) \in \om^2$ for which $d(x_i,A) < p - \pq{s}$.
 
This preceding implications can be proved round-robin style easily. We just note that for the first left-to-right-hand implication one chooses some positive rational number $\pq{s}$ with $2 \cdot \pq{s} < p - d(x,A)$ and uses the density of $(x_i)_{i \in \om}$.

Suppose now that $N \geq 1$. Let $\ep \in \del$, $\ca{A}$ and $f_d$ be as in Lemma \ref{lemma transition from closed code to distance from convex hull}. We define $I \subseteq \om^4 \times \baire$ by
\begin{align*}
I(i,s,n,m,\alpha) 
\iff& \ \fcode^\ca{\R^N}(\alpha) \cap [-m,m]^N \neq \emptyset \ \& \  d(\dense{\R^N}{i},H(\fcode^{\R^N}(\alpha) \cap [-m,m]^N)) < 2^{-n} - \pq{s}\\
\iff& \ (m,\alpha) \in \ca{A} \ \& \ f_d(m,\alpha,\dense{\R^N}{i}) < 2^{-n} - \pq{s}
\end{align*}
so that from the preceding claim
\begin{align*}
& \ \set{x \in \R^N}{d\left(x,H\left(\fcode^{\R^N}(\alpha) \cap [-m,m^N]\right)\right) < 2^{-n}}=\\
& \hspace*{75mm} = \ \cup_{(i,s) \in I(n,m,\alpha)} B_{\R^N}(\dense{\R^N}{i},\pq{s}) \\
& \hspace*{75mm} = \ \cup_{(i,s) \in I(n,m,\alpha)} N(\R^N, \langle i,s \rangle).
\end{align*}
Therefore we define $f: \om^2 \times \baire \to \baire$
\[
f(n,m,\alpha)(j)
=
\begin{cases}
\langle i,s \rangle, & \ \text{if} \ j = \langle i,s \rangle \ \text{and} \ I(i,s,n,m,\alpha),\\
0, & \ \text{else},
\end{cases}
\]
and we have that
\begin{align*}
\set{x \in \R^N}{d\left(x,H\left(\fcode^{\R^N}(\alpha) \cap [-m,m^N]\right)\right) < 2^{-n}} 
=& \ \cup_{j \in \om} N(\R^N, f(n,m,\alpha)(j))\\
=& \ \cup_{j \in \om} \overline{N(\R^N, f(n,m,\alpha)(j))}.
\end{align*}
In particular $f(n,m,\alpha)$ is an open code of the former set. From Lemma \ref{lemma transition from closed code to distance from convex hull} and the fact that $(i \in \om \to \dense{\R^N}{i})$ is recursive it is immediate that $I$ is a $\Sigma^0_4(\ep)$ set; hence $f$ is $\Sigma^0_5(\ep)$-recursive. This settles assertion (a) of the proposition.

Regarding (b) we recall the proof that a convex open set $U \subseteq \R^N$ is convexly generated: we consider a sequence $(B_i)_{i \in \om}$ of (bounded) open balls with $U = \cup_i B_i = \cup_i \overline{B_i}$ and we have that $U = \cup_i H(\cup_{j \leq i} \overline{B_j})$.

Since the set $\set{x \in \R^N}{d\left(x,H\left(\fcode^{\R^N}(\alpha) \cap [-m,m^N]\right)\right) < 2^{-n}}$ is convex, from the preceding arguments we have
\[
\set{x \in \R^N}{d\left(x,H\left(\fcode^{\R^N}(\alpha) \cap [-m,m^N]\right)\right) < 2^{-n}} = \cup_i H\left(\cup_{j \leq i} \overline{N(\R^N, f(n,m,\alpha)(j))}\right).
\]
Clearly the latter union is increasing. Thus we need to find a function $g$ of the required complexity such that $\rfn{g(n,m,\alpha)^\ast}(i)$ is a closed code for the set $H\left(\cup_{j \leq i} \overline{N(\R^N, f(n,m,\alpha)(j))}\right)$ for all $i$.

To do this we apply Lemma \ref{lemma transition from closed code to code of convex hull}. It is easy to find a recursive function $h_0: \om^3 \times \baire \to \om$ such that $\normR{x}_{\infty} \leq h_0(n,m,i,\alpha)$ for all $x \in \cup_{j \leq i} \overline{N(\R^N, f(n,m,\alpha)(j))}$ and all $n,m,i,\alpha$.

We also need to find a recursive function $h_1$ such that the set $\cup_{j \leq i} \overline{N(\R^N, f(n,m,\alpha)(j))}$ equals to $\fcode^{\R^N}(h_1(n,m,i,\alpha))$ for all $n,m,i,\alpha$. To do this we notice first that for an open ball $B(x,p)$ in $\R^N$ it holds
\begin{align*}
y \not \in \overline{B(x,p)}
\iff& \ p < d(x,y)\\
\iff& \ (\exists m, k)[p + 2^{-k} < d(x,\dense{\R^N}{m}) \ \& \ y \in B(\dense{\R^N}{m},2^{-k})],
\end{align*}
so
\begin{align}\nonumber
y \not \in \cup_{j \leq i} \overline{B(x_j,p_j)}
\iff& \ (\forall j \leq i)(\exists m, k)[p_j + 2^{-k} < d(x_j,\dense{\R^N}{m}) \ \& \ y \in B(\dense{\R^N}{m},2^{-k})]\\ \nonumber
\iff& \ (\exists s)[(s)_0, (s)_1 \in \Seq \ \& \ \lh((s)_0) = \lh((s)_1) = i+1\\
&\hspace*{-20mm} \& \ (\forall j < \lh((s)_0))[p_j + 2^{-((s)_1)_j} < d(x_j,\dense{\R^N}{((s)_0)_j}) \ \& \ y \in B(\dense{\R^N}{((s)_1)_j},2^{-((s)_1)_j})]].\label{equation finite union closure}
\end{align}
We will replace $x_j$ with the center of $N(\R^N,f(n,m,\alpha)(j))$ and $p_j$ with the radius of the latter ball.

The last of the preceding quantifications is clearly the finite intersection of open balls. We fix a recursive function $cp$ such that
\[
\cap_{j< {s'}}N(\R^N,(u)_{j}) = \cup_t N(\R^N,cp(u,s',t)), \quad u \in \Seq,
\]
see \cite[3B.2]{yiannis_dst}. 

We also fix $n,m,i,\alpha$ and $s$ with $(s)_0, (s)_1 \in \Seq, \lh((s)_0) = \lh((s)_1) = i+1$ for the discussion, and moreover a recursive function $h_1'$ such that
\[
N(\R^N,h_1'(s,j)) = B(\dense{\R^N}{((s)_1)_j},2^{-((s)_1)_j})
\]
for all $j$. It is then clear that the finite intersection
\begin{align*}
& \ \cap_{j < \lh((s)_0)} \set{B(\dense{\R^N}{((s)_1)_j},2^{-((s)_1)_j})}{\pq{(f(n,m,\alpha)(j))_1}+ 2^{-((s)_1)_j}\\
& \hspace*{55mm} < d(\dense{\R^N}{(f(n,m,\alpha)(j))_0},\dense{\R^N}{((s)_0)_j})}
\end{align*}
equals to
\[
\cap_{j < \lh((s)_0)} N(\R^N,(h''_1(n,m,s,i,\alpha))_j)
\]
where $h_1''$ is a recursive function such that
\[
(h_1''(n,m,s,i,\alpha))_j =
\begin{cases}
h_1'(s,j), & \text{if} \ \pq{{(f(n,m,\alpha)(j))_1}}+ 2^{-((s)_1)_j} \\
          & \hspace*{15mm} < d(\dense{\R^N}{(f(n,m,\alpha)(j))_0},\dense{\R^N}{((s)_0)_j})\\
& \text{and} \ j < \lh((s)_0)\\
0, & \text{else}.
\end{cases} 
\]
Using the key property of the function $cp$ the preceding finite intersection equals to
\[
\cup_t N(\R^N,cp(h_1''(n,m,s,i,\alpha),\lh((s)_0),t)).
\]
So we define
\begin{align*}
h_1(n,m,i,\alpha)(\langle s, t\rangle) = cp(h_1''(n,m,s,i,\alpha),\lh((s)_0),t)
\end{align*}
if $(s)_0, (s)_1 \in \Seq, \lh((s)_0) = \lh((s)_1) = i+1$, and we let $h_1(n,m,i,\alpha)$ be $0$ in any other case. Using (\ref{equation finite union closure}) it becomes clear that
\[
\R^N \setminus \left(\cup_{j \leq i} \overline{N(\R^N, f(n,m,\alpha)(j))}\right)
= \cup_{e} N(\R^N,h_1(n,m,i,\alpha)(e)).
\]
Now, given $h_0$, $h_1$ as above we have
\begin{align*}
\cup_{j \leq i} \overline{N(\R^N, f(n,m,\alpha)(j))} 
=& \ \fcode^{\R^N}(h_1(n,m,i,\alpha))\\
=& \ \fcode^{\R^N}(h_1(n,m,i,\alpha)) \cap [-h_0(n,m,i,\alpha), h_0(n,m,i,\alpha)]^N. 
\end{align*}
Hence
\begin{align*}
H(\cup_{j \leq i} \overline{N(\R^N, f(n,m,\alpha)(j))}) 
=& \ H \left(\fcode^{\R^N}(h_1(n,m,i,\alpha)) \cap [-h_0(n,m,i,\alpha), h_0(n,m,i,\alpha)]^N \right)\\
=& \ \fcode^{\R^N}(v( h_0(n,m,i,\alpha),  h_1(n,m,i,\alpha)))\\
=& \ \fcode^{\R^N}(h(n,m,i,\alpha)),
\end{align*}
where $v$ is as in Lemma \ref{lemma transition from closed code to code of convex hull} and  $h(\alpha,n,m,i,k) = v( h_0(n,m,i,\alpha),  h_1(n,m,i,\alpha))(k)$. Clearly $h$ is $\Sigma^0_5(\ep)$-recursive.

Thus we need find a $\Sigma^0_5(\ep)$-recursive function $g$ such that $\rfn{g(n,m,\alpha)^\ast}(i) = h(n,m,i,\alpha)$ and $g(n,m,\alpha)(0) = 1$. (Notice that from the preceding equalities the union $\cup_{i \in \om}\fcode^{\R^N}(\rfn{g(n,m,\alpha)^\ast}(i))$ is increasing.) Such a function $g$ is obtained by standard recursive-theoretic methods:

Recall that $\rfn{\beta}^{\om, \baire}(i) \downarrow$ exactly when there is a unique $y \in \baire$ such that for all $s \in \om$, we have that $y \in N(\baire,s) \iff \scode_{1}(\beta,y,i,s)$. In the latter case the value of $\rfn{\beta}(i)$ is the unique $y$ as before.  

Hence we need to have
\begin{align}\label{equation B}
h(n,m,i,\alpha) \in N(\baire,s) \iff \scode_1(g(n,m,i,\alpha)^\ast,n,m,i,s)
\end{align}
for all $n,m,i,\alpha,s$. 

Clearly the relation $R^h(n,m,i,\alpha,s) \iff h(n,m,i,\alpha) \in N(\baire,s)$ is $\Delta^0_5(\ep)$. We recall that every $\Delta^0_5(\gamma)$ subset of the naturals is recursive on the fourth Turing jump $\gamma^{(4)}$ of $\gamma$, \emph{uniformly} on $\gamma$.

Thus there are a recursive points $\ep_0, \ep_1 \in \baire$ and a recursive function $S$ such that
\begin{align}\nonumber
h(n,m,i,\alpha) \in N(\baire,s)
\iff& \ R^h(n,m,i,\alpha,s) \\ \nonumber
\iff& \ \rfn{\ep_0}^{\baire \times \om^4,\om}((\ep \oplus \alpha)^{(4)},n,m,i,s) = 1\\ \nonumber
\iff& \ \scode^{\baire \times \om^4}_1(\ep_1,(\ep \oplus \alpha)^{(4)},n,m,i,s) \\
\iff& \ \scode^{\om^4}_1(S(\ep_1,(\ep \oplus \alpha)^{(4)}),n,m,i,s) \label{equation C}
\end{align}
for all $n,m,i,\alpha,s$. We thus take
\[
g(n,m,\alpha)(0) = 1 \ \ \text{and} \ \ g(n,m,\alpha)(t+1) = S(\ep_1,(\ep \oplus \alpha)^{(4)})(t).
\]
Clearly $g$ is $\Sigma^0_5(\ep)$-recursive. From (\ref{equation C}) it follows that
\[
h(n,m,i,\alpha) \in N(\baire,s) \iff \scode_1(S(\ep_1,(\ep \oplus \alpha)^{(4)}),n,m,i,s) \iff \scode_1(g(n,m,\alpha)^\ast,n,m,i,s),
\]
\ie we have established the key property (\ref{equation B}).\smallskip

This concludes the proof of Proposition \ref{proposition basis of coding}.

\subsection{Convexly generated codes and the proof to the uniform Preiss Separation Theorem}

\label{subsection convexly generated codes and the proof to the uniform Preiss Separation Theorem}

We fix a natural number $N \geq 1$. The \emph{hierarchy} $(\cgclass^N_\xi)_{\xi < \om_1}$ \emph{of the family $\cgclass^N$ of all convexly generated subsets of $\R^N$} is given by
\begin{align*}
\cgclass^N_0 =& \ \set{K \subseteq \R^N}{K \ \text{is compact and convex}}\\
\cgclass^N_\xi =& \ \set{\cup_{i \in \om}\cap_{j \in \om} A_{ij}}{\text{for all $i,j$ there is $\xi_{ij} < \xi$ such that $A_{ij} \in \cgclass^N_{\xi_{ij}}$}\\
& \hspace*{30mm} \text{and for all $i\leq i'$ we have $\cap_{j \in \om}A_{ij} \subseteq \cap_{j \in \om}A_{i'j}$}},
\end{align*}
where $\xi \geq 1$. 

By induction one can verify that $\cgclass^N_\eta \subseteq\cgclass^N_\xi$ for all $\eta < \xi$ and that
\begin{align*}
\cgclass^N 
=& \ \text{the family of all convexly generated subsets of $\R^N$}\\
=& \ \cup_{\xi < \om_1} \cgclass^N_\xi.
\end{align*}

The family $\cgcode^N$ of \emph{convexly generated codes} is defined simultaneously with the coding function $\cgcf^N: \cgcode^N \surj \cgclass^N$ by recursion\footnote{We need to define them simultaneously in order to be able to express the increasing union of already encoded sets. Notice that unlike the Borel or semi-positive codes, the family of convexly generated codes depends on the space that we are considering.}
\begin{align*}
\cgcode^N_0 =& \ \set{\alpha}{\alpha(0) = 0  \ \& \ \fcode^{\R^N}(\alpha^\ast) \ \text{is compact and convex}}\;,\\
\cgcfun{0}^N:& \ \cgcode^N_0 \surj \cgclass^N_0: \alpha \mapsto \fcode^{\R^N}(\alpha^\ast)\;,\\
\cgcode^N_\xi =& \ \set{\alpha}{\alpha(0) = 1 \ \& \ (\forall i,j)(\exists \eta < \xi)[\rfn{\alpha^\ast}(\langle i,j \rangle) \in \cgcode^N_\eta\\
& \ \& \ \text{for all $i \leq i'$ we have} \ \cap_j \cgcfun{\eta}(\rfn{\alpha^\ast}(\langle i, j)) \subseteq \cap_j \cgcfun{\eta}(\rfn{\alpha^\ast}(\langle i', j))]}\;,\\
\cgcfun{\xi}^N:& \ \cgcode^N_\xi \surj \cgclass^N_\xi: \alpha \mapsto \cup_{i \in \om}\cap_{j \in \om} \cgcfun{\eta(\alpha,i,j)}(\rfn{\alpha^\ast}(\langle i,j \rangle))\\
& \ \text{where $\eta(\alpha,i,j) =$ the least $\eta$ such that} \ \rfn{\alpha^\ast}(\langle i,j \rangle) \in \cgcode^N_\eta\;,\\
\cgcode^N =& \ \cup_{\xi < \om_1} \cgcode^N_\xi\;,\\
\cgcf^N=& \ \cup_{\xi < \om} \cgcfun{\xi}^N.
\end{align*}
It can be proved by induction that for all $1 \leq \eta < \xi$ the restriction of $\cgcfun{\xi}^N$ on $\cgclass^N_\eta$ coincides with $\cgcfun{\eta}^N$, and thus $\cgcf^N$ is well-defined. 

Given $\alpha \in \cgcode^N$ we put
\[
\normcg{\alpha}^N = \text{the least $\xi < \om_1$ such that} \ \alpha \in \cgcode^N_\xi,
\]
and then it is clear that $\cgcf^N(\alpha) = \cgcfun{\normcg{\alpha}^N}^N(\alpha)$ for all $\alpha \in \cgcode^N$. We say that an $\alpha \in \cgcode^N$ is a \emph{code for the convexly generated} set $A \subseteq \R^N$ if $A = \cgcf^N(\alpha)$.\bigskip 

Now we can \emph{prove Theorem \ref{theorem uniform preiss separation}.}\bigskip

This is done by combining our constructive proof to Theorem \ref{theorem preiss separation} with the method of the Suslin-Kleene Theorem \cite[7B.3, 7B.4]{yiannis_dst} by replacing the term ``recursive" with ``$\del$-recursive". We only need to make sure that the codes of the separating sets $D^\sigma_{(k,l,j)}$ and the clauses defining the tree $J$, are obtained as \del-functions of the codes $\alpha,\beta$.

The analysis following the constructive proof to Theorem \ref{theorem preiss separation} deals exactly with these points. We consider the following:
\begin{list}{$\bullet$}{}
\item The function $f_d$ of Lemma \ref{lemma transition from closed code to distance from convex hull}.

\item The recursive function $\strans^{\R^N} \equiv \strans$ as in the statement of Proposition \ref{proposition from analytic code to good Souslin code}, so that $\strans(\alpha)$ is a good Souslin code for $\scode^{\R^N}_1(\alpha)$.
\item The functions $f$, $g$ as in the statement of Proposition {proposition basis of coding}.
\end{list}

Below we follow the notation of our proof to Theorem \ref{theorem preiss separation}.

The first clause in the definition of the tree $J$,
\[
C_1(\alpha,m,s) \iff Q_{\dec{s}}(\alpha) \cap [-m,m]^N \neq \emptyset,
\] 
where $(Q_{\dec{s}}(\alpha))_{s \in \Seq}$ is the good Souslin scheme encoded by $\strans(\alpha)$, is a $\del$ relation. This is because of the compactness of closed bounded subsets of $\R^N$,
\begin{align*}
C_1(\alpha,m,s)
\iff& \ (\exists x)[\fcode^{\R^N}((\strans(\alpha))_s,x) \ \& \ x \in [-m,m]^N]\\
\iff& \ (\exists x \in \del(\alpha))[\fcode^{\R^N}((\strans(\alpha))_s,x) \ \& \ x \in [-m,m]^N].
\end{align*}
For the second clause, let $\tilde{S}(\beta)$ be the tree of pairs, which corresponds to the analytic set $\pi^{-1}[\scode^{\R^N}_1(\beta)]$. Clearly we can obtain a code for $\tilde{S}(\beta)$ recursively in $\beta$. Hence the second clause,
\[
C_2(\beta,s,t) \iff \lh(\dec{s}) = \lh(\dec{t}) \ \& \ (\dec{s},\dec{t}) \in \tilde{S}(\beta),
\]
is actually a $\Delta^0_1$ relation.

Regarding the third clause we consider the condition
\[
d\left(h(b),H\left(Q_{u}(\alpha) \cap [-m,m]^N\right)\right) < 2^{-\lh(b)+3}
\]
where $d$ is the distance with respect to the supremum norm $Q_u(\alpha)$ is as above and $b$, $u$ are finite sequences of the same length. So we define
\begin{align*}
C_3(\alpha,\beta,\dec{s},\dec{t}) 
\iff& \ \lh(s) = \lh(t) \ \& \ d(h(\dec{s}), H(\fcode^{\R^N}((\strans(\alpha))_{t}) \cap [-m,m]^N) < 2^{-
\lh(\dec{s})+3}\\
\iff& \ \lh(s) = \lh(t) \ \& \ f_d(m,(\strans(\alpha))_{t},h(\dec{s})) < 2^{-
\lh(\dec{s})+3}.
\end{align*}
Since $f_d$ is $\del$-recursive it follows that $C_3$ is \del.

It remains to obtain codes for the sets $D^\sigma_{(k,l,j)}$ in terms of $\alpha$ and $\beta$. The latter sets are either the empty set or the whole space, or according to (\ref{equation definition of Dsigma}),
\begin{align*}
D^\sigma_{(k,l,j)}(\alpha,\beta) 
=& \ \set{x \in \R^N}{d\left(x,H(Q_{\cn{u}{(j)}}(\alpha)) \cap [-m,m]^N\right) < 2^{-\lh(u)}},
\end{align*}
where $u$ is obtained from $\sigma$. (The case distinction is done in a \del-way.) 

It is clear that we can find recursive points $\gamma_0$, $\gamma_1$, $\delta_0$, $\delta_1$ such that $\gamma_i$ and $\delta_i$ are Borel and convexly generated codes for $A_i$ respectively, where $A_0 = \emptyset$ and  $A_1 = \R^N$.

Regarding the third case we use the functions $f$, $g$ from Proposition \ref{proposition basis of coding}. Let $v: \om \times \om \to \om$ be recursive such that whenever $s \in \om$ encodes a tuple $\sigma = (m,b,d,u)$ then $v(s,j)$ is a code for $\cn{u}{(j)}$, \ie $\dec{v(s,j)} = \cn{u}{(j)}$. Then for such $s$, $u$, $\sigma$ we have
\[
D^\sigma_{(k,l,j)}(\alpha,\beta)  = \set{x \in \R^N}{d\left(x,H(\fcode^{\R^N}((\strans(\alpha))_{v(s,j)})) \cap [-m,m]^N\right) < 2^{-\lh(\dec{v(s,j)})+1}}.
\]
Therefore $f(\lh(\dec{v(s,j)})-1,m,(\strans(\alpha))_{v(s,j)})$ is a Borel code for $D^\sigma_{(k,l,j)}(\alpha,\beta)$. To find a convexly generated code we need to make a minor change to the function $g$. From Proposition \ref{proposition basis of coding} we have that
\[
D^\sigma_{(k,l,j)}(\alpha,\beta) = \cup_{i \in \om} \fcode^{\R^N}(\rfn{g(\lh(\dec{v(s,j)})-1,m,(\strans(\alpha))_{v(s,j)})^\ast}(i)),
\]
and that the preceding union is increasing. We need to bring the latter union in the form $\cup_i \cap_k A_{ik}$. Clearly there is a \del-recursive function $\tilde{g}$ such that $\tilde{g}(n,m,\gamma)(0) = 1$ and $\rfn{\tilde{g}(n,m,\gamma)^\ast}(\langle i, k \rangle) = \rfn{g(n,m,\gamma)^\ast}(i)$ for all $i,k$. Hence
\begin{align*}
D^\sigma_{(k,l,j)}(\alpha,\beta) =& \ \cup_{i \in \om} \cap_k \fcode^{\R^N}(\rfn{\tilde{g}(\lh(\dec{v(s,j)})-1,m,(\strans(\alpha))_{v(s,j)})^\ast}(\langle i, k \rangle))\\
= & \ \cgcf(\tilde{g}(\lh(\dec{v(s,j)})-1,m,(\strans(\alpha))_{v(s,j)})
\end{align*}
and so $\tilde{g}(\lh(\dec{v(s,j)})-1,m,(\strans(\alpha))_{v(s,j)})$ is a convexly generated code for $D^\sigma_{(k,l,j)}(\alpha,\beta)$.

Thus we can find partial $\del$-recursive functions $d_i(\ep,\alpha,\beta,s,k,l,j)$, $i=1,2$, such that whenever $\scode_1^{\R^N}(\alpha)$ is convexly generated and disjoint from $\scode_1^{\R^N}(\beta)$, and $s \in \om$ encodes a tuple $\sigma = (m,b,d,u)$ then $d_1(\ep,\alpha,\beta,s,k,l,j)$ and $d_2(\ep,\alpha,\beta,s,k,l,j)$ are defined and are Borel and convexly generated codes for $D^\sigma_{(k,l,j)}(\alpha,\beta)$ respectively.

The proof proceeds exactly as the one of \cite[7B.3]{yiannis_dst} by considering $\del$-recursive functions instead of recursive ones.

\subsection{\HYP-convexly-generated is almost \del \ and convexly generated.}

Now we can effectivize the notion of a convexly generated set as we did before with the semi-positive sets (see Definition \ref{definition effective semipositive}).

\begin{definition}\normalfont
\label{definition effective convexly generated}
Let $\ep \in \baire$ and $N \geq 1$. The family $\speffclass^N_\xi(\ep)$ of all \emph{$\HYP(\ep)$-convexly generated} subsets of $\R^N$ of \emph{order $\xi$} is defined by
\[
\cgeffclass^N_\xi(\ep) = \set{\cgcf^N(\alpha)}{\alpha \ \text{is $\ep$-recursive and} \ \normcg{\alpha}^N = \xi}.
\]
The family of all \emph{$\HYP(\ep)$-convexly generated subsets} of $\R^N$ is
\[
\cup_{\xi} \ \cgeffclass^N_\xi(\ep).
\]

It is clear that a set is convexly generated exactly when it is $\HYP(\ep)$-convexly generated for some $\ep \in \baire$. 
\end{definition}\smallskip

Finally we \emph{prove Corollary \ref{corollary effective Preiss}.}\bigskip

The equivalence between $(i)$ and $(iii)$ is clear from the Preiss Separation Theorem.

$(iii) \Longrightarrow (ii)$ We consider the \del-recursive function $v$ in the statement of Theorem \ref{theorem uniform preiss separation}. If $A$ is $\del$ then there are recursive points $\alpha, \beta \in \baire$, which are analytic codes for $A$ and $\R^N \setminus A$ respectively. If additionally $A$ is convex then from Theorem \ref{theorem uniform preiss separation} the $\del$-point $v(\alpha,\beta)$ is a convexly generated code for $A$.

$(ii) \Longrightarrow (iii)$ Clearly a convexly generated set is convex, so we need to show that every convexly generated set with a $\del$ (convexly generated) code is \del. This is proved exactly as in \cite[7B.5]{yiannis_dst} by actually showing the stronger uniform result, that there is a recursive function $u = (u_1,u_2): \baire \to \baire \times \baire$ such that for all $\alpha \in \cgcode$ the points $u_1(\alpha)$, $u_2(\alpha)$ are analytic codes for $\cgcf(\alpha)$ and $\R^N \setminus \cgcf(\alpha)$ respectively. 

Note that in the basis step of the definition of $u$ we need a recursive function $g_0 = (g_{01},g_{02}): \baire \to \baire \times \baire$ such that $g_{01}(\alpha)$ and $g_{02}(\alpha)$ are analytic codes for the sets $\fcode^{\R^N}(\alpha)$ and $\R^N \setminus \fcode^{\R^N}(\alpha) = \ocode^{\R^N}(\alpha)$ respectively. The definition of $g_{01}$ is easy to arrange. Regarding $g_{02}$, since $\ocode^{\R^N}(\alpha) = \cup_{\n} N(\R^N,\alpha(n))$, we need a function $h: \om \to \baire$ such that $h(s)$ is an analytic code for $N(\R^N,s)$, and a function $\vee_\om: \baire \to \baire$ such that $\vee_{\om}((\alpha)_n)$ is an analytic code for $\cup_{\n} \scode^{\R^N}_1((\alpha)_n)$. Both these functions are easy to construct. We omit the details.

\section*{Acknowledgments}
                 
The author is currently a Fellow of the \textit{Programme 2020 researchers : Train to Move}\includegraphics[scale=0.01]{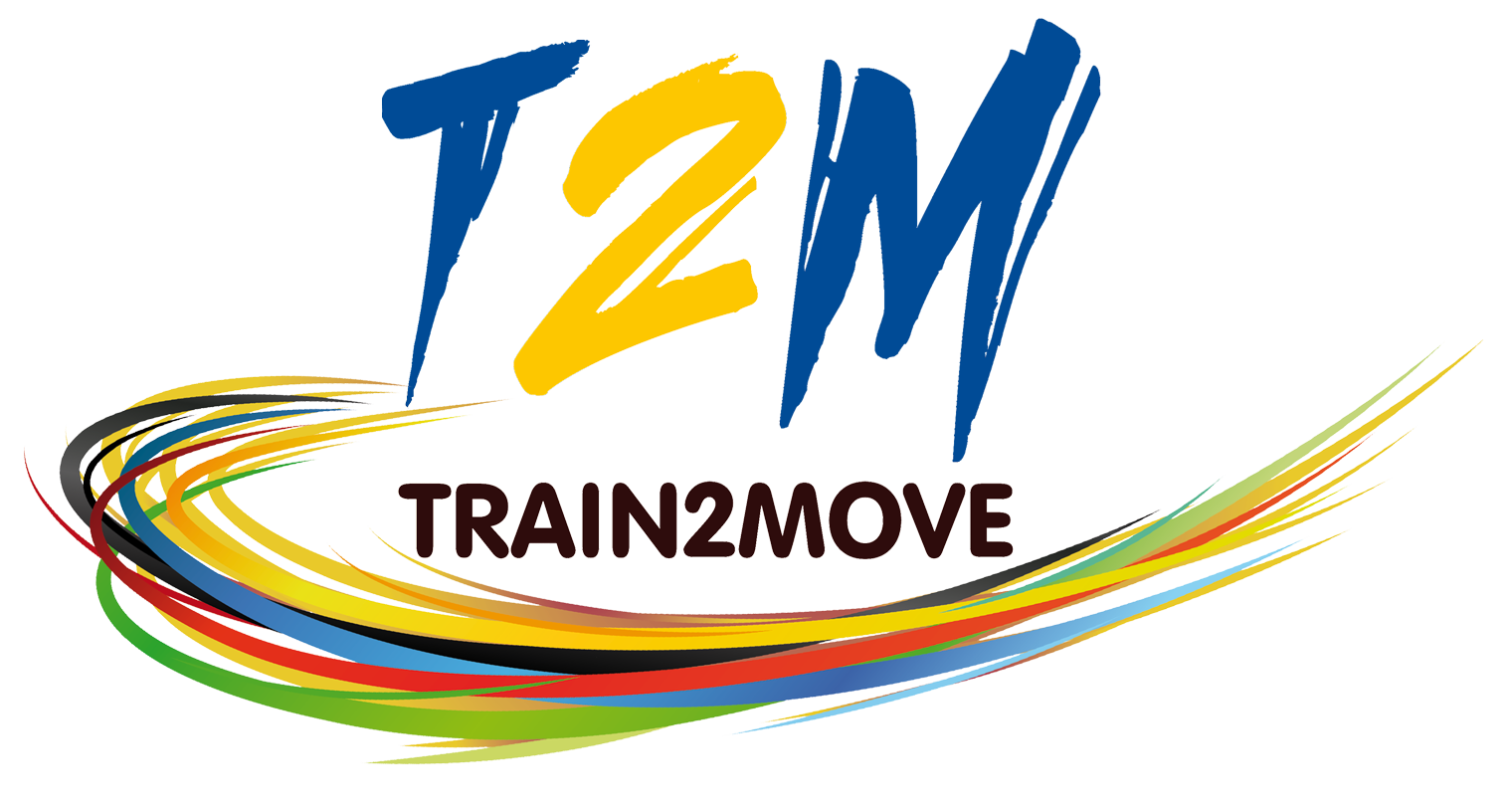} 
at the Mathematics Department ``Guiseppe Peano" of the University of Turin, Italy. The part of the article, which concerns the Dyck separation was written, while the author was a Scientific Associate at TU Darmstadt, Germany.

Thanks are owed to {\sc A. Kechris}, {\sc D. Lecomte }, and {\sc B. Miller} for helpful discussions.

\end{document}